\documentclass{article}
\usepackage{graphicx} 
\usepackage{amssymb,amsthm}
\usepackage{amsmath,amscd}

\usepackage{microtype}
\usepackage{hyperref}
\usepackage{mfl}
\usepackage{enumitem}
\usepackage{cite, tikz}

\title{Subvarieties of pointed Abelian $\ell$-groups}
\author{Filip Jankovec\footnote{Institute of Computer Science, Czech Academy of Sciences,
Pod Vodárenskou věží 271/2
182 00 Prague 8
Czech Republic.} 
\footnote{Department of Algebra, Faculty of Mathematics and Physics, Charles University,
Sokolovská 49/83, 
186 75 Prague 8
Czech Republic.}}

\newcommand{\mfrak}[1]{\mathfrak{#1}}

\newtheorem{thm}{Theorem}[section]
\newtheorem{lemma}[thm]{Lemma}

\newtheorem{defn}[thm]{Definition}
\newtheorem{cor}[thm]{Corollary}

\newcommand{\A}{\alg{A}}  
\newcommand{\B}{\alg{B}} 
 \renewcommand{\C}{\alg{C}} 

\newcommand{\fa}{\ensuremath{\mathtt{f}}}

\newcommand{\U}{\mfrak{U}}

\renewcommand{\Alg}{\ensuremath{\mathbf{Alg}}}
\renewcommand{\div}{\operatorname{div}}

\renewcommand{\R}{\ensuremath{\alg{R}}}
\renewcommand{\Q}{\ensuremath{\alg{Q}}}
\renewcommand{\Z}{\ensuremath{\alg{Z}}}
\renewcommand{\ltimes}{\overrightarrow{\times}}
\renewcommand{\val}{\ensuremath{\mathbf{val}}}

\date{}

\begin{document}

\maketitle
\begin{abstract}
This paper provides a complete classification of all subvarieties of pointed Abelian lattice-ordered groups ($\ell$-groups), as well as all subquasivarieties that are generated by their totally ordered members. We present two complementary approaches to achieve this classification.

First, using purely $\ell$-group-theoretic methods, we analyze the structure of lexicographic products and values to identify all join-irreducible members of the lattice of subvarieties of positively pointed Abelian $\ell$-groups. We provide a novel equational basis for each of these subvarieties, leading to a complete description of the entire subvariety lattice. As a direct application, our $\ell$-group-theoretic classification yields an alternative, self-contained proof of Komori's classification of subvarieties of MV-algebras.

Second, we explore the connection to MV-algebras via an extended version of Mundici's $\Gamma$ functor.  We prove that this functor preserves universal classes, a result of independent model-theoretic interest. This allows us to lift the classification of universal classes of totally ordered MV-algebras, due to Gispert, to a complete classification of universal classes of totally ordered pointed Abelian $\ell$-groups. As a direct consequence, we obtain a complete structural description of the lattice of subquasivarieties that are generated by their totally ordered members. 
\end{abstract}


\section{Introduction}

Lattice-ordered groups ($\ell$-groups) represent a foundational class of algebraic structures with a rich history of celebrated results \cite{Hahn:HahnTheorem,Glass-Holland:l-groups,Glass:PartiallyOrderedGroups, Birkhoff:LatticeTheory, Darnel:TheoryLGroups}. The field remains an active area of research, largely because the lattice of subvarieties of $\ell$-groups is notoriously complex and not yet fully understood \cite{Kopytov:Lattice-groups}. Even the simplest non-trivial case, the variety of Abelian $\ell$-groups, presents a deep mathematical challenge.

MV-algebras form another extremely important class of algebraic structures. Originally introduced as the algebraic semantics for Łukasiewicz logic \cite{Lukasiewicz-Tarski:Untersuchungen}, these algebras have been the subject of extensive study for decades \cite{Cignoli-Ottaviano-Mundici:AlgebraicFoundations,Gabbay-Metcalfe:ContinuousUninorms,DiNola-Leustean:Handbook, Mundici-LogicUlamGame}. One of the most essential tools in their study is the celebrated Mundici functor, a categorical equivalence between MV-algebras and Abelian $\ell$-groups with a strong unit. This functor provides a powerful bridge between the two fields, allowing many fundamental results in the theory of MV-algebras to be obtained by translating problems into the more established context of $\ell$-groups.

However, the theory of MV-algebras is in some respects richer than that of Abelian $\ell$-groups, primarily because its language contains an additional designated constant. A similar level of expressive power can be achieved in the theory of Abelian $\ell$-groups by augmenting their language with an additional constant, which gives rise to the class of pointed Abelian $\ell$-groups. As this article demonstrates, this class allows one to reason about MV-algebras while remaining entirely within the realm of Abelian $\ell$-groups. Furthermore, this class forms a variety, making it a natural and well-behaved algebraic framework.
This contrasts with the class of Abelian $\ell$-groups with a strong unit, which is not even first-order definable.

In addition to its algebraic significance, our work is also strongly motivated by logic. Pointed Abelian $\ell$-groups serve as the algebraic models for pointed Abelian logic \cite{Cintula-Jankovec-Noguera:SuperabelianLogics}, a system that can be viewed as a meeting-point of Łukasiewicz logic \cite{Lukasiewicz-Tarski:Untersuchungen} and Abelian logic \cite{Casari:ComparativeLogics,Meyer-Slaney:AbelianLogic}. This paper provides a complete classification of all axiomatic and semilinear finitary extensions of pointed Abelian logic. Further details on the logical perspective can be found in \cite{Cintula-Jankovec-Noguera:SuperabelianLogics}.

The starting point for our investigation is Komori's foundational classification of the varieties of MV-algebras \cite{Komori:SuperLukasiewiczPropositional}. This was later connected to the $\ell$-group setting by Young \cite{Young:Varieties_of_pointed_Abelian_l-groups}, who used the Mundici functor to establish a correspondence between these varieties and varieties of positively pointed Abelian $\ell$-groups. This paper generalizes and extends Young's work in several key aspects:

\begin{enumerate}
    \item First, we broaden the scope from positively pointed Abelian $\ell$-groups to the entire class of pointed Abelian $\ell$-groups.

    \item Second, we develop our theory of varieties semantically from first principles within the theory of Abelian $\ell$-groups, rather than relying on the Mundici functor and the theory of MV-algebras. This contrasts with the more syntactic approach of \cite{Young:Varieties_of_pointed_Abelian_l-groups} and allows us to prove more general theorems (e.g., compare our Theorem~\ref{t:strong_unit} with \cite[Lemma 4.6]{Young:Varieties_of_pointed_Abelian_l-groups}). Moreover, this strategy provides us with more structural insight, and we believe readers will find our reasoning more structurally transparent than that presented in \cite{Young:Varieties_of_pointed_Abelian_l-groups}. As a consequence of our framework, we derive Komori's classification of subvarieties of MV-algebras in Theorem~\ref{t:Komori}.

    \item Third, while the work in \cite{Young:Varieties_of_pointed_Abelian_l-groups} established the correspondence between positively pointed Abelian $\ell$-groups and MV-algebras and provided a syntactic translation to convert MV-algebraic identities into identities of positively pointed $\ell$-groups, the resulting translated equations are unnecessarily complex and lack direct $\ell$-group theoretic intuition. This paper improves upon this by providing explicit, native, and simplified equational bases directly in the language of $\ell$-groups, analogous to Komori’s original description of the varieties of MV-algebras in \cite{Komori:SuperLukasiewiczPropositional}.

    \item Finally, we expand the scope of investigation beyond varieties, which was the sole focus of Young's work, to universal classes and quasivarieties. Specifically, we use Gispert's classification of universal classes of totally ordered MV-algebras and quasivarieties of MV-algebras generated by their totally ordered members from \cite{Gispert:UniversalClassesofMV-chains} to provide the corresponding classification for pointed Abelian $\ell$-groups.
    
\end{enumerate}


The paper is structured as follows. Section \ref{s:Preliminaries} introduces the necessary preliminaries from universal algebra and the theory of Abelian $\ell$-groups.
Section \ref{s:Decomposition} develops the core technical results concerning lexicographic decompositions of finitely generated totally ordered $\ell$-groups, culminating in Corollary \ref{c:minimal variety} and Theorem~\ref{t:strong_unit}. 
Section \ref{s:Semantics} provides a semantic characterization of the varieties of pointed Abelian $\ell$-groups, presenting novel and more elegant proofs for some known results (e.g., Lemma~\ref{l:Non-SimpleGroupsRank}).
The central contribution of the paper is presented in Section \ref{s:Axiomatization}, which gives a complete classification and axiomatization of all subvarieties of pointed Abelian $\ell$-groups. 
Finally, Section \ref{s:Mundici} applies our results to the theory of MV-algebras, providing an alternative proof of Komori's classification. Moreover, by using Gispert's classification of universal classes of totally ordered MV-algebras, we provide a similar classification for totally ordered pointed Abelian $\ell$-groups and consequently obtain a classification of all quasivarieties generated by totally ordered pointed Abelian $\ell$-groups.

\section{Preliminaries} \label{s:Preliminaries}

In this section, we introduce some basic terminology and well-known results about Abelian $\ell$-groups. 
We assume the reader has a basic background in universal algebra, but no advanced prerequisites are required to follow the rest of the paper.
 We will use the standard notations $\mathbf{H, I, S,P,\PU}$
 to denote closure under homomorphisms, isomorphisms, subalgebras, products, and ultraproducts, respectively. We will write $\A \in \K_{\FSI}$ iff the algebra $\A \in \K$ and $\A$ is finitely subdirectly irreducible.

First let us recall Birkhoff's Theorem and its variants.

\begin{thm}[Birkhoff \cite{Birkhoff:AbstractAlgebras,Birkhoff:SubdirectUnions}] \label{t:Birkhoff}
    Let $\K$ be a class of algebras of the same signature.
The following conditions are equivalent:

\begin{enumerate} 
\item $\K=\HSP(\K)$. 

\item $\K$ is the class of all models of some theory, all of whose axioms are equations.
\end{enumerate}
Moreover, if algebras in $\K$ have a group reduct, we can add the following:
\begin{enumerate} \setcounter{enumi}{2}
\item $\K$ is the class of all models of some theory, all of whose axioms are equations with the right-hand side equal to $0$.
\end{enumerate}

Furthermore, if $\K$ satisfies any of the equivalent conditions above (i.e., if $\K$ is a variety), then:
$$\K=\ISP(\K_{\FSI}).$$
\end{thm}

Let $\A$ be an algebra, $\rul{eq}$ an equation, and $e$ an evaluation. We write $\A \vDash_e \rul{eq}$ to denote that $\rul{eq}$ \emph{is satisfied} in $\A$ under the evaluation $e$. Conversely, $\A \nvDash_e \rul{eq}$ indicates that $\rul{eq}$ \emph{is not satisfied} under $e$.

An equation $\rul{eq}$ is said to \emph{hold} in an algebra $\A$, denoted $\A \vDash \rul{eq}$, if it is satisfied for all possible evaluations. Conversely, if there is at least one evaluation for which $\rul{eq}$ is not satisfied, it does \emph{not hold} in $\A$, and we write $\A \nvDash \rul{eq}$.

We now provide the definition of an essential notion used in this paper.

\begin{defn}    
Let $\K \cup \{\A\}$ be a class of algebras of the same language $\lang{L}$. Let $F$ be a finite subset of $A$ and $\B \in \K$.
We say that a mapping $f_F:F \rightarrow B$ is a \emph{partial embedding} of the set $F$ from $\A$ into $\B$ 
if $f_F$ is a one-to-one mapping and for every operation symbol $\lambda \in \lang{L}$ of arity $n$, and for all  $a_1,\dots,a_n \in F$ such that $\lambda^{\A}(a_1,\dots,a_n) \in F$ we have 
\begin{equation*}
    f_F(\lambda^{\A}(a_1,\dots,a_n))=\lambda^{\B}(f_F(a_1),\dots,f_F(a_n)).
\end{equation*}
We say that $\A$ is partially embeddable into the class $\K$ if for each finite set $F \subseteq A$ there is a $\B \in \K$ and there is a partial embedding $f_F$ of $F$ into $\B$.
If $\A$ is partially embeddable into the singleton class $\{\B\}$, we simply say $\A$ is partially embeddable into $\B$.
\end{defn}

\begin{thm}[Mal'tcev \cite{Malcev:AlgebraicSystems}] \label{t:partembISPU}
Let $\K$ be a class of algebras of the same signature.
The following conditions are equivalent:

\begin{enumerate} 
\item $\K$ is the class of all models of some theory, all of whose axioms are universal formulas.
\item $\K=\ISPU(\K)$. 
\item $\K$ is closed under partial embeddings (i.e.\ for any algebra $\A$ that is partially embeddable into $\K$ we have $\A \in \K$).
\end{enumerate}
    
\end{thm}

Now we will focus on Abelian $\ell$-groups.

\begin{defn}
An algebra $\A=\tuple{A,+,-,\lor,\land,0}$ is an \emph{Abelian $\ell$-group} if $\tuple{A,+,-,0}$ is an Abelian group, $\tuple{A,\lor,\land}$ is a lattice and $\A$ satisfies the monotonicity condition: for all $x,y,z\in A$, if $x \leq y$, then $x+z \leq y+z$. A subalgebra of an Abelian $\ell$-group is referred to as an \emph{$\ell$-subgroup}. 
\end{defn}

It is well-known (see, e.g.\ \cite[Chapter V]{Fuchs:PartiallyOrdered}), that the defining conditions of Abelian $\ell$-groups can be expressed by means of equations, so they form a variety which we denote by $\mathbb{AL}$.
Also, it is well-known that all Abelian $\ell$-groups are torsion-free.
 
We will denote by $\Z$ and $\R$ the $\ell$-groups of integers and reals with the underlying universes $\mathbb{Z}$ and $\mathbb{R}$, taken under their natural order.

For $a \in A$, we define the \emph{absolute value} $|a| := a \lor -a$ and we let $n \cdot a:=a+\cdots+a$ denote the $n$-fold sum for every $n \in \mathbb{N}$ (with $0 \cdot a := 0$). For $z \in \mathbb Z {\setminus} \mathbb N$ we set ${z \cdot a:=(-z) \cdot (-a)}$.

\begin{defn}
An $\ell$-subgroup $\B$ of an Abelian $\ell$-group $\A$ is called \emph{convex} if it satisfies the following:

\[
\forall b \in B, \forall a \in A, \quad \big( |a| \leq |b| \implies a \in B \big).
\]
\end{defn}
It is well known that congruences on Abelian $\ell$-groups are in one-to-one correspondence with their convex $\ell$-subgroups (often referred to as \emph{$\ell$-ideals}); see, e.g.,  \cite{Fuchs:PartiallyOrdered}.

The following results are also well-known.

\begin{thm}[Clifford \cite{Clifford:PartiallyOrderedAbelianGroups}] \label{t:Clifford}
An Abelian $\ell$-group is finitely subdirectly irreducible if and only if it is totally ordered.
\end{thm}

\begin{thm}[Gurevich, Kokorin \cite{Gurevich-Kokorin:UniversalEquivalence}]  \label{t:Gurevich-Kokorin}\label{t:UnThrAbGroup}
    All non-trivial totally ordered Abelian $\ell$-groups are universally equivalent.
    Equivalently, for every non-trivial totally ordered Abelian $\ell$-group $\A$ we have $\ISPU(\A)=\ISPU(\Z)$.
\end{thm}

From Theorems \ref{t:Clifford} and \ref{t:Gurevich-Kokorin} one can easily derive the following:

\begin{thm}[Khisamiev \cite{Khisamiev:UniversalTheoryAbelianGroups}] \label{t:quasivarietyAbGroups} \label{t:Khisamiev}
The quasivariety of all Abelian $\ell$-groups is generated by $\Z$. Equivalently, $\ISPPU(\Z)=\mathbb{AL}$.
\end{thm}
 
\begin{thm}[Hölder's Theorem \cite{Holder:Axiome}] \label{t:Holder}
    The following are equivalent for any totally ordered Abelian $\ell$-group $\A$:

    \begin{enumerate}
        \item $\A$ embeds into $\R$.
        \item $\A$ is Archimedean.
        \item $\A$ is simple.
    \end{enumerate}
\end{thm}

\begin{lemma} \label{l:about rationals} 
Let $\A$ be a non-trivial $\ell$-subgroup of
$\R$. Then $A$ is a dense subset of $\mathbb{R}$ (with respect to the standard topology on $\mathbb{R}$) or $\A \cong \Z$.
\end{lemma}
\begin{proof}

    Since multiplication by any strictly positive element of $\mathbb{R}$ is an $\ell$-group automorphism on $\R$ we can without loss of generality assume $1 \in A$. If $A$ also contains some $\xi \in \mathbb{R} \setminus \mathbb{Q}$, we can recursively generate a strictly decreasing sequence of elements of $A$ converging to $0$ as follows. We start with $r_1=1$ and $r_1=\xi-\lfloor \xi\rfloor \cdot1$. For $n> 0$, we define $k_n=\lfloor \frac{r_{n-1}}{r_n} \rfloor$ and let $r_{n+1}=r_{n-1}-k_n \cdot r_n$. Hence, $A$ must be a dense subset of $\mathbb{R}$.

    Otherwise, $\A \in \S(\Q)$.
    By \cite[Theorem 2]{Beaumont-Zuckerman:rationals} and \cite[Corollary 2]{Beaumont-Zuckerman:rationals} for every non-cyclic $\A \in \S(\Q)$ and for each $0<\epsilon \in \mathbb{R}$ there is an $a \in A$ such that $0<a \leq \epsilon$. Since $\A$ is also closed under addition, $A$ must be a dense subset of $\mathbb{R}$.
    If $\A$ is cyclic, then $\A \cong \Z$.
\end{proof}



\begin{defn}
For an Abelian $\ell$-group $\A=\tuple{A,+,-,\lor,\land,0}$ and $a \in A$ we define a \emph{pointed} Abelian $\ell$-group $\A_a=\tuple{A,+,-,\lor,\land,0,\fa}$,\footnote{It should be stressed that the choice of the symbol $\fa$ has no algebraic motivation; its origin is purely logical. The symbol is chosen to align with its use as a 'falsum' constant (representing falsehood) in logical systems. This becomes relevant when we consider pointed Abelian $\ell$-groups as algebraic counterparts to expansions of Abelian logic, as detailed in \cite{Meyer-Slaney:AbelianLogic,Cintula-Jankovec-Noguera:SuperabelianLogics}.}
where $\fa^{\A_a}=a$. We denote the class of pointed Abelian $\ell$-groups by $p\mathbb{AL}$. We say that $\A_a \in p\mathbb{AL}$ is \emph{positively} pointed if $a \geq 0$, \emph{negatively} pointed if $a \leq 0$, and \emph{$0$-pointed} if $a=0$. We denote these classes by $p\mathbb{AL}^+, p\mathbb{AL}^-$ and $p\mathbb{AL}^0$.
\end{defn}

Clearly, the classes $p\mathbb{AL}, p\mathbb{AL}^+, p\mathbb{AL}^-$ and $p\mathbb{AL}^0$ are varieties. 
Observe that for an Abelian $\ell$-group $\A$ and $a \in A$ we have that $\A$ is finitely subdirectly irreducible iff $\A_a$ is finitely subdirectly irreducible. Therefore, using Theorem~\ref{t:Clifford} we obtain that a pointed Abelian $\ell$-group is finitely subdirectly irreducible iff it is totally ordered.

As a corollary of Theorems \ref{t:Clifford} and \ref{t:Birkhoff}, together with the fact that every algebra embeds into an ultraproduct of its finitely generated subalgebras (which follows directly from Theorem \ref{t:partembISPU}), we obtain the following.

\begin{cor} \label{c:Birkhoff++}
    Any subvariety of $p\mathbb{AL}$ is generated by its finitely generated totally ordered members.
\end{cor}

Furthermore, the variety $p\mathbb{AL}^0$ is clearly term equivalent to $\mathbb{AL}$. 
We now establish the following result.




\begin{lemma} \label{l:about R1}
\begin{enumerate}
    
    \item $p\mathbb{AL}_{\FSI}=\ISPU(\R_{-1},\R_0,\R_1)$.
    \item $p\mathbb{AL}_{\FSI}^+=\ISPU(\R_0,\R_1)$.
    \item $p\mathbb{AL}_{\FSI}^-=\ISPU(\R_{-1},\R_0)$.
    \item $p\mathbb{AL}_{\FSI}^0=\ISPU(\R_{0})=\ISPU(\Z_{0})$.

    \item $p\mathbb{AL}=\HSP(\R_{-1},\R_0,\R_1)$.
    \item $p\mathbb{AL}^+=\HSP(\R_0,\R_1)$.
    \item $p\mathbb{AL}^-=\HSP(\R_{-1},\R_{0})$.
    \item $p\mathbb{AL}^0=\HSP(\R_{0})=\HSP(\Z_{0})$.
\end{enumerate}

\end{lemma}

\begin{proof}
Let us recall that pointed Abelian $\ell$-groups are finitely subdirectly irreducible iff they are totally ordered.

The inclusions $\supseteq$ are trivial. We prove the remaining inclusions in a convenient order.
\begin{enumerate}
    \item[4.] Follows, since the class $p\mathbb{AL}_{\FSI}^0$ is term equivalent to $\mathbb{AL}_{\FSI}.$
 
    \item 
    Let us fix an arbitrary $\A_a \in p\mathbb{AL}_{\FSI}$.
    Using Theorem~\ref{t:Gurevich-Kokorin} and \cite[Lemma 4.3]{Cintula-Jankovec-Noguera:SuperabelianLogics} we obtain $\A_a \in \ISPU(\{\R_b \mid b \in \mathbb{R}\})$. Since for $b \in \mathbb{R}$ we have $\R_b \cong \R_{-1}$, $\R_b \cong \R_{0}$, or $\R_b \cong \R_{1}$ we obtain $\A_a \in \ISPU(\R_{-1},\R_0,\R_1)$. 
    
    \item 
    For $\A_a \in p\mathbb{AL}_{\FSI}^0$ we already know $\A_a \in \ISPU(\R_0,\R_1)$.
    Therefore, we fix arbitrary $\A_a \in p\mathbb{AL}_{\FSI}^+ \setminus p\mathbb{AL}_{\FSI}^0$. By the previous point we know that $\A_a \in \ISPU(\R_{-1},\R_0,\R_1)$.
    By \cite[Theorem 5.6]{Bergman:UniversalAlgebra} we have $$\ISPU(\R_{-1},\R_0,\R_1)=\ISPU(\R_{-1})\cup \ISPU(\R_{0}) \cup \ISPU(\R_{1}).$$ Since $\A_a \in p\mathbb{AL}_{\FSI}^+ \setminus p\mathbb{AL}_{\FSI}^0$, it follows $a>0$. Therefore, $\A_a \notin \ISPU(\R_{-1})$ since $\A_a \nvDash \fa \leq 0$. Thus $\A_a \in \ISPU(\R_0,\R_1)$.

    \item Can be proved similarly as the previous point.
\end{enumerate}
    
The other four points follow directly by applying Theorem~\ref{t:Birkhoff}.
\end{proof}

We will later prove a strengthening of the second part of Lemma~\ref{l:about R1} in Lemma~\ref{l:about R2}.





\section{The lexicographic decompositions} \label{s:Decomposition}

This section focuses on understanding the structure of the lexicographic product of $\ell$-groups. The lexicographic product is a key operation on $\ell$-groups. It allows for the creation of otherwise unintuitive $\ell$-groups that play an essential role in various classifications (see e.g.\ the famous Hahn Theorem from \cite{Hahn:HahnTheorem}). Since we will need to use this tool frequently in the following sections, here we establish several basic properties of this construction. 
The main result of this section is Theorem~\ref{t:strong_unit}, which is a stronger version of \cite[Lemma 4.6]{Young:Varieties_of_pointed_Abelian_l-groups}, which demonstrates that when classifying universal classes of totally ordered pointed Abelian $\ell$-groups, we can restrict our focus to those that are strongly pointed or $0$-pointed.
We will proceed to the definition of the lexicographic product.

\begin{defn}
    Let $\A$ be a totally ordered Abelian $\ell$-group and $\B$ be an Abelian $\ell$-group.
    We define the Abelian $\ell$-group $\A \ltimes \B$ as the Abelian $\ell$-group $\A \times \B$ with modified lattice operations:

    $$\tuple{a_1,b_1} \lor \tuple{a_2,b_2}=
    \begin{cases}
        \tuple{a_1,b_1} &  a_1 > a_2\\
        \tuple{a_2,b_2} & a_1 < a_2 \\
        \tuple{a_1,b_1 \lor b_2} & a_1=a_2
    \end{cases}
    $$

    and
    $$\tuple{a_1,b_1} \land \tuple{a_2,b_2}=
    \begin{cases}
        \tuple{a_1,b_1} &  a_1 < a_2\\
        \tuple{a_2,b_2} & a_1 > a_2 \\
        \tuple{a_1,b_1 \land b_2} & a_1=a_2
    \end{cases}
    $$

\end{defn}

Note that for totally ordered Abelian $\ell$-groups $\A,\B$ and an Abelian $\ell$-group $\C$ it holds that $\A \ltimes \B$ is totally ordered and moreover $(\A \ltimes \B) \ltimes \C \cong \A \ltimes (\B \ltimes \C)$. Therefore, we will commonly omit parentheses and we will just write $\A \ltimes \B \ltimes \C$.



\begin{lemma} \label{l:decomposition_of_fin_gen_lgroups}
    Let $\A$ be a finitely generated totally ordered Abelian $\ell$-group and $\B$ be a convex $\ell$-subgroup of $\A$.
    Then $\A \cong (\A/\B) \ltimes \B$.
\end{lemma}

\begin{proof}
    
    Since $\A$ is finitely generated, $\A/\B$ is also finitely generated. Since $\A/\B$ is a finitely generated torsion-free Abelian group, it is group-isomorphic (but not necessarily $\ell$-group isomorphic) to a free Abelian group (see \cite[Theorem 10.19]{Rotman:Groups}). Thus the group exact sequence $0\rightarrow \B \rightarrow^\iota \A \rightarrow^\pi \A/\B \rightarrow 0$ splits and there is a group homomorphism $p:\A/\B \rightarrow \A$ such that $\pi \circ p=id_{\A/\B}$. It is well-known (see \cite[Lemma 10.3]{Rotman:Groups}) that the mapping $\varphi:(\A/\B) \times \B \rightarrow \A$ defined as $\tuple{x,y} \mapsto p(x)+\iota(y)$ is a group isomorphism.
    We want to show that $\varphi$ is an $\ell$-group isomorphism between $(\A/\B) \ltimes \B$ and $\A$ as well. 
    We will show that $\varphi$ is order-preserving. 
    Let us pick $\tuple{a,b}, \tuple{c,d} \in (A/B) \times B$ such that $\tuple{a,b} \leq \tuple{c,d}$. 
    We distinguish two cases:

    \begin{enumerate}
        \item If $a<c$, we obtain $\pi(p(a)+\iota(b)) < \pi(p(c)+\iota(d))$ since \\
        $\begin{aligned}
      \pi(p(a)+\iota(b)) &= \pi(p(a))+\pi(\iota(b)) = a< c  \\
      &= \pi(p(c))+\pi(\iota(d)) = \pi(p(c)+\iota(d)).
      \end{aligned}$
    
    Since $\pi$ is order preserving and $\A$ is totally ordered we derive from $\pi(p(a)+\iota(b)) < \pi(p(c)+\iota(d))$ that $p(a)+\iota(b) < p(c)+\iota(d)$.
    \item If $a=c$, then $b \leq d$, so clearly $p(a)+\iota(b) \leq p(c)+\iota(d)$.
    \end{enumerate}

    This proves that $\varphi$ is an $\ell$-group homomorphism. It remains to prove that $\varphi^{-1}$ preserves ordering as well. This follows easily from the fact that $\A$ is a totally ordered algebra and $\varphi$ is an order preserving isomorphism. Thus $\varphi$ is an $\ell$-group isomorphism between $(\A/\B) \ltimes \B$ and $\A$. 
\end{proof}

Generally, it is well known that one can define a lexicographic ordering on a totally ordered Abelian group $\A$, provided one has a short exact sequence of Abelian groups $0\rightarrow \B \rightarrow \A \rightarrow \A/\B \rightarrow 0$ and $\B, \A/\B$ are totally ordered \cite[Problem 1.8]{Clay-Rolfsen:OrderedGroups}. However, one still needs to assume that $\A/\B$ is finitely generated to prove that the sequence $0\rightarrow \B \rightarrow \A \rightarrow \A/\B \rightarrow 0$ splits. Let us also remark that one can easily prove a pointed version of this lemma by just verifying that the isomorphism in the proof of Lemma~\ref{l:decomposition_of_fin_gen_lgroups} preserves the point structure.

\begin{lemma} \label{l:DecomPointLgrp}
    Let $\A_b$ be a finitely generated totally ordered pointed Abelian $\ell$-group and $\B_b$ be a pointed convex $\ell$-subgroup of $\A_b$. 
    Then $\A_b \cong (\A/\B)_0 \ltimes \B_b$.
\end{lemma}

\begin{defn}
    Let $\A_a \in p\mathbb{AL}$. We say $a$ is a \emph{strong unit} of $\A$ if for each $b \in A$ there is $z \in \mathbb{Z}$ such that $z \cdot a \geq b$.
    We say that a pointed Abelian $\ell$-group $\A_a$ is \emph{strongly pointed}, whenever $a$ is a strong unit of $\A$. 
\end{defn}

We note that our definition of a strong unit deviates from the standard convention in the literature, which typically requires strong units to be strictly positive. By allowing negative strong units, we gain a distinct structural advantage: for every totally ordered pointed Abelian $\ell$-group $\A_b$ with $b \neq 0$, there is a unique strongly pointed convex $\ell$-subgroup, which is (as a convex $\ell$-subgroup) generated by $b$.

\begin{lemma} \label{l:lex1}
    Let $\A_a$ be a non-trivial pointed Abelian $\ell$-group. Moreover, assume there exists a strong unit of $\A$.
    Then $\Z_0 \ltimes \A_a \in \ISPU(\A_a)$.
\end{lemma}

\begin{proof}
    Let us denote a strong unit of $\A$ by $b$.
    Consider an ultrapower of $\A_a$ defined as $\C=\prod_{i \in \omega} \A_a /\U$, where $\U$ is a non-principal ultrafilter on $\omega$ and let $\iota:d \mapsto \tuple{d,\dots,d}$  be the canonical embedding of $\A$ into $\C$. Consider the element $c \in C$ defined as $c=\tuple{|b|,2|b|,3|b|,\dots}$. Since the set $\{n \cdot |b| \geq |d|\}$ is cofinite for any $d \in \A$, we have $\{n \mid n \cdot |b| \geq |d|\} \in \U$ and thus  $c > |\iota(d)|$ for any $d \in \A$.
    Let us denote by $\B$ the $\ell$-subgroup of $\C$ generated by $\iota[A] \cup \{c\}$. For any $d \in A$ we have $-c<\iota(d)<c$ thus we can uniquely express any element of $B$ as $n \cdot c+\iota(d)$ for some $n \in \mathbb{Z}$ and $d \in A$. For any $n_1,n_2 \in \mathbb{Z}$ and $d_1,d_2 \in A$ we have $n_1 \cdot c+\iota(d_1) \geq n_2 \cdot c+\iota(d_2)$ iff $n_1 > n_2$ or $n_1=n_2$ and $d_1 \geq d_2$. This proves $\B_{\iota(a)} \cong \Z_0 \ltimes \A_a$. Since $\B_{\iota(a)} \in \ISPU(\A_a)$, we conclude $\Z_0 \ltimes \A_a \in \ISPU(\A_a)$.
\end{proof}

\begin{cor} \label{c:minimal variety}
    $p\mathbb{AL}^0$ is the smallest nontrivial subvariety of $p\mathbb{AL}$. Alternatively, we can say that any non-trivial proper subvariety of $p\mathbb{AL}$ contains $p\mathbb{AL}^0$ as a subvariety.
\end{cor}

\begin{proof}
    Let $\K$ be a nontrivial subvariety of $p\mathbb{AL}$. Take any non-trivial totally ordered $\A_a \in \K$. Then $\A_a$ contains a pointed $\ell$-subgroup isomorphic to $\Z_1$, $\Z_0$ or $\Z_{-1}$. 
    If we get $\Z_1 \in \K$ or $\Z_{-1} \in \K$, by Lemma~\ref{l:lex1} we get $\Z_0 \ltimes \Z_1 \in \K$ or $\Z_0 \ltimes \Z_{-1} \in \K$.
    Since $\K$ is closed under $\mathbf H$, we get $\Z_0 \in \K$ in all cases.
    Therefore, $\Z_0 \in \K$ and by Lemma~\ref{l:about R1} we conclude that $p\mathbb{AL}^0 \subseteq \K$.
\end{proof}

Using this corollary we get the stronger version of Lemma~\ref{l:about R1}.

\begin{lemma} \label{l:about R2}
    \begin{enumerate}
    \item $p\mathbb{AL}=\HSP(\R_{-1},\R_1)$.
    \item $p\mathbb{AL}^+=\HSP(\R_1)$.
    \item $p\mathbb{AL}^-=\HSP(\R_{-1})$.
    \item $p\mathbb{AL}^0=\HSP(\R_{0})=\HSP(\Z_{0})$.
\end{enumerate}
\end{lemma}

We will need to understand how partial embedding interacts with lexicographic products.
\begin{lemma} 
    Let $\A,\B,\C,\alg{D}$ be totally ordered Abelian $\ell$-groups and assume $\A$ partially embeds into $\B$.
    Then $\C \ltimes \A \ltimes \alg{D}$ partially embeds into $\C \ltimes \B \ltimes \alg{D}$.
    
    Equivalently, $\A \in \ISPU(\B)$ implies $\C \ltimes \A \ltimes \alg{D} \in \ISPU(\C \ltimes \B \ltimes \alg{D})$.
\end{lemma}

\begin{proof}
    Let $\{\varphi_F\}_{F \subseteq A, |F|<\omega}$ be a family of partial embeddings from $\A$ to $\B$. Let us denote by $\pi_A$ the projection from $\C \ltimes \A \ltimes \alg{D}$ to $\A$.
    We claim that $\{\psi_{G}\}_{G \subseteq C \times A \times D, |G|<\omega}$, where $\psi_G:\tuple{c,a,d} \mapsto \tuple{c,\varphi_{\pi_A[G]}(a),d}$, is a family of partial embeddings from $\C \ltimes \A \ltimes \alg{D}$ to $\C \ltimes \B \ltimes \alg{D}$.

    Let us fix $G$ a finite subset of $C \times A \times D$. We show $\psi_G$ is a partial embedding.
    Assume $\tuple{c_1,a_1,d_1},\tuple{c_2,a_2,d_2}, \tuple{c_1+c_2,a_1+a_2,d_1+d_2} \in G$. Consequently $a_1,a_2,a_1+a_2 \in \pi_A[G]$ and we have 

    \begin{align*}
        \psi_G(\tuple{c_1,a_1,d_1})+\psi_G(\tuple{c_2,a_2,d_2})&=\tuple{c_1,\varphi_{\pi_A[G]}(a_1),d_1}+\tuple{c_2,\varphi_{\pi_A[G]}(a_2),d_2}\\
        &=\tuple{c_1+c_2,\varphi_{\pi_A[G]}(a_1)+\varphi_{\pi_A[G]}(a_2),d_1+d_2}\\&=\tuple{c_1+c_2,\varphi_{\pi_A[G]}(a_1+a_2),d_1+d_2}\\
        &=\psi_G(\tuple{c_1+c_2,a_1+a_2,d_1+d_2}).
    \end{align*}

    In a similar fashion we can check that if $\tuple{c,a,d},\tuple{-c,-a,-d} \in G$ we have $-\psi_G\tuple{c,a,d}=\psi_G\tuple{-c,-a,-d}$ and always we have $\psi_G\tuple{0,0,0}=\tuple{0,0,0}$.

    Therefore, $\psi_G$ preserves addition, subtraction and zero constant. It remains to check $\psi_G$ preserves operations $\lor,\land$. Since $\A,\B,\C,\alg{D}$ are totally ordered, it is enough to show that $\psi_G$ preserves the lattice  ordering. 
        Let us assume that $\tuple{c_1,a_1,d_1},\tuple{c_2,a_2,d_2} \in G$ and  $\tuple{c_1,a_1,d_1} <\tuple{c_2,a_2,d_2}$. Let us distinguish three cases:

    \begin{enumerate}
        \item If $c_1<c_2$ we get
        \begin{align*}
            \psi_G(\tuple{c_1,a_1,d_1})&=\tuple{c_1,\varphi_{\pi_A[G]}(a_1),d_1}\\&< \tuple{c_2,\varphi_{\pi_A[G]}(a_2),d_2}=\psi_G(\tuple{c_2,a_2,d_2}).
        \end{align*}
        
        \item If $c_1=c_2$ and $a_1<a_2$ we get $\varphi_{\pi_A[G]}(a_1)<\varphi_{\pi_A[G]}(a_2)$ (by injectivity and order-preservation of $\varphi_{\pi_A[G]}$) and thus we get 
        \begin{align*}
            \psi_G(\tuple{c_1,a_1,d_1})&=\tuple{c_1,\varphi_{\pi_A[G]}(a_1),d_1}\\&< \tuple{c_1,\varphi_{\pi_A[G]}(a_2),d_2}=\psi_G(\tuple{c_2,a_2,d_2}).
        \end{align*}
        
        \item If $c_1=c_2, a_1=a_2$ and $d_1<d_2$ we get
        \begin{align*}
            \psi_G(\tuple{c_1,a_1,d_1})&=\tuple{c_1,\varphi_{\pi_A[G]}(a_1),d_1}\\ &< \tuple{c_1,\varphi_{\pi_A[G]}(a_1),d_2}=\psi_G(\tuple{c_2,a_2,d_2}).
        \end{align*}
    \end{enumerate}
    
    This shows $\psi_G$ is indeed a partial embedding of $G$ into $\C \ltimes \B \ltimes \alg{D}$. Since the finite set $G$ was arbitrary, we get that
    $\C \ltimes \A \ltimes \alg{D}$ partially embeds into $\C \ltimes \B \ltimes \alg{D}$.
\end{proof}

We can easily observe that the lemma also holds for pointed Abelian $\ell$-groups.

\begin{lemma} \label{l:lexproducts_partembeddings}
    Let $\A_a,\B_b,\C_c,\alg{D}_d$ be totally ordered pointed Abelian $\ell$-groups and assume $\A_a$ partially embeds into $\B_b$.
    Then $\C_c \ltimes \A_a \ltimes \alg{D}_d$ partially embeds into $\C_c \ltimes \B_b \ltimes \alg{D}_d$.
    Equivalently, $\A_a \in \ISPU(\B_b)$ implies $$\C_c \ltimes \A_a \ltimes \alg{D}_d \in \ISPU(\C_c \ltimes \B_b \ltimes \alg{D}_d).$$
\end{lemma}

Finally, we conclude this section with the following theorem which is a strengthening of \cite[Lemma 4.6]{Young:Varieties_of_pointed_Abelian_l-groups}.

\begin{thm} \label{t:strong_unit}
     Let $\A_b$ be a finitely generated totally ordered pointed Abelian $\ell$-group with $b \neq 0$ and $\B_b$ be its strongly pointed convex $\ell$-subgroup.
    Then $\ISPU(\A_b)=\ISPU(\B_b)$.
\end{thm}

\begin{proof}

Clearly, $\B_b \in \S(\A_b)$ and thus $\B_b \in \ISPU(\A_b).$

\hspace{-1.8mm} We prove the other implication. \hspace{-1.8mm}
By Lemma~\ref{l:lex1} we have $\Z_0 \ltimes \B_b \in \ISPU(\B_b)$. Also, by Lemma~\ref{l:about R1} we have $(\A/\B)_0 \in \ISPU(\Z_0)$ and by Lemma~\ref{l:lexproducts_partembeddings} we have $(\A/\B)_0 \ltimes \B_b \in \ISPU(\Z_0 \ltimes \B_b)$. By Lemma~\ref{l:DecomPointLgrp} we get $(\A/\B)_0 \ltimes \B_b \cong \A_b$. Therefore we showed $\A_b \in \ISPU(\B_b)$, which completes the proof.
\end{proof}

Theorem~\ref{t:strong_unit} tells us that the universal theory of any totally ordered pointed Abelian $\ell$-group is equal to the universal theory generated by its strongly pointed convex $\ell$-subgroup or the universal theory generated by $\Z_0$.
In other words, when classifying universal classes of totally ordered pointed Abelian $\ell$-groups, we can restrict our focus to those which are strongly pointed or $0$-pointed.

From Theorem \ref{t:strong_unit} and Corollary \ref{c:minimal variety} we obtain the following two additional corollaries.

\begin{cor} \label{c:strong unit variety}
    Every non-trivial subvariety of $p\mathbb{AL}$ that is not equal to $p\mathbb{AL}^0$ is generated by its finitely generated totally ordered strongly pointed members.
\end{cor}

\begin{cor} \label{c:strong unit quasi}
    Every non-trivial subquasivariety of $p\mathbb{AL}$ generated by its totally ordered members is also generated by those of its finitely generated totally ordered members that are either strongly pointed or $0$-pointed.
\end{cor}

\section{Characterization of subvarieties generated by a finitely generated totally ordered pointed Abelian \texorpdfstring{$\ell$-group}{l-group}} \label{s:Semantics}

In this section we describe all join-irreducible subvarieties of pointed Abelian $\ell$-groups. 
Although one could obtain the result of this section using Theorem~\ref{t:strong_unit} from Section \ref{s:Decomposition}, the Mundici functor and the Komori classification, we have chosen a different, more self-contained approach that does not rely on the theory of MV-algebras.
However, a reader familiar with Komori classification of MV-algebras will find some of the proofs here possibly familiar, since they often use similar techniques (for comparison see \cite{Cignoli-Ottaviano-Mundici:AlgebraicFoundations}).

\begin{lemma} \label{l:iso-lex}
    Let $\A_a,\B_b,\C_c$ be totally ordered pointed Abelian $\ell$-groups and $\alg{D}_d$ be a pointed Abelian $\ell$-group.
    Let $\psi:\A_a \rightarrow \B_b$ be an injective homomorphism.
    Then $\varphi:\C_c \ltimes \A_a \ltimes \alg{D}_d \rightarrow \C_c \ltimes \B_b \ltimes \alg{D}_d$ defined as $\tuple{x,y,z} \mapsto \tuple{x,\psi(y),z}$ is an injective homomorphism as well.

    Moreover, if $\psi$ is an isomorphism then also $\varphi$ is an isomorphism.

\end{lemma}

\begin{proof}
    First, $\tuple{x,y,z} \mapsto \tuple{x,\psi(y),z}$ is clearly an injective homomorphism of groups $\C_c \times \A_a \times \alg{D}_d$ and $\C_c \times \B_b \times \alg{D}_d$. 
    
    We check that $\varphi$ preserves the lattice operations of the lexicographic product.
    
    For each $x_1,x_2 \in C, y_1,y_2 \in A$ and $z_1, z_2 \in D$ we have

    \begin{align*}
   \varphi(\tuple{x_1,y_1,z_1}) \lor \varphi(\tuple{x_2,y_2,z_2})&=\tuple{x_1,\psi(y_1),z_1} \lor \tuple{x_2,\psi(y_2),z_2}=\\
       &=\begin{cases}
         \tuple{x_1,\psi(y_1),z_1} & \text{if }x_1 >x_2 \\
         \tuple{x_2,\psi(y_2),z_2} & \text{if }x_1<x_2 \\
         \tuple{x_1,\psi(y_1),z_1} & \text{if }x_1=x_2, y_1>y_2 \\
         \tuple{x_2,\psi(y_2),z_2} & \text{if }x_1=x_2, y_1<y_2 \\
         \tuple{x_1,\psi(y_1),z_1 \lor z_2 } & \text{if }x_1=x_2, y_1=y_2 \\
    \end{cases} \\ &=
    \begin{cases}
         \varphi(\tuple{x_1,y_1,z_1}) & \text{if }x_1 >x_2 \\
         \varphi(\tuple{x_2,y_2,z_2}) & \text{if }x_1<x_2 \\
         \varphi(\tuple{x_1,y_1,z_1}) & \text{if }x_1=x_2, y_1>y_2 \\
         \varphi(\tuple{x_2,y_2,z_2}) & \text{if }x_1=x_2, y_1<y_2 \\
         \varphi(\tuple{x_1,y_1,z_1 \lor z_2}) & \text{if }x_1=x_2, y_1=y_2 
    \end{cases} \\&=
    \varphi(\tuple{x_1,y_1,z_1} \lor \tuple{x_2,y_2,z_2}).
   \end{align*}

    This shows $\varphi$ preserves $\lor$. Similarly, we can show that $\varphi$ preserves $\land$. Therefore, $\varphi$ is an injective homomorphism.

    In the case when $\psi$ is an isomorphism, there exists the inverse isomorphism $\psi^{-1}:\B_b \rightarrow \A_a$. By the previous part of the proof $\varphi^{-1}:\C_c \ltimes \B_b \ltimes \alg{D}_d \rightarrow \C_c \ltimes \A_a \ltimes \alg{D}_d$ defined as $\tuple{x,y,z} \mapsto \tuple{x,\psi^{-1}(y),z}$ is an injective homomorphism. Clearly, $\varphi^{-1}$ is also inverse to $\varphi$, which shows $\varphi$ must be an isomorphism.
\end{proof}

\begin{lemma} \label{l:continuous_functions}
    Let $a_1,\dots,a_m \in \mathbb R$ and let $f:(\R \ltimes \R)^m \rightarrow \R \ltimes \R$ be a pointed Abelian $\ell$-group term function. Let $\pi_2:\R \ltimes \R \rightarrow \R$ be a projection to the second coordinate.
    Then the $m$-ary function 
    $$g(x_1,\dots,x_m):=\pi_2( f(\tuple{a_1,x_1},\dots,\tuple{a_m,x_m}))$$ is a continuous function from $\R^m$ to $\R$.
    In particular, every pointed Abelian $\ell$-group term function on $\R$ is continuous.
\end{lemma}

\begin{proof}
    We show that $g$ is a composition of continuous functions. We proceed by induction on the complexity of the term $f$.
    Clearly, constants and projections are continuous functions. Moreover,
    addition and subtraction are defined component-wise, thus they are continuous functions. 

    It remains to check maximum and minimum.
    For any constants $b_1,b_2 \in \mathbb R$ we define 
    
    $$x_1 \lor_{b_1,b_2} x_2=   
    \begin{cases}
            x_1 & \text{if } b_1>b_2  \\
            x_2 & \text{if } b_1<b_2  \\
            x_1 \lor x_2 & \text{if } b_1=b_2
    \end{cases}$$ \\ \text{ and } 
    $$x_1 \land_{b_1,b_2} x_2= 
     \begin{cases}
            x_2 & \text{if } b_1>b_2  \\
            x_1 & \text{if } b_1<b_2  \\
            x_1 \land x_2 & \text{if } b_1=b_2
    \end{cases}  
    .$$
    
Thus for all $b_1,b_2 \in \mathbb R$ we obtain that $\lor_{b_1,b_2}$ and $\land_{b_1,b_2}$ are continuous functions.
Since $g$ is a composition of addition, subtraction, constants and functions $\lor_{b_i,b_j}$ and $\land_{b_i,b_j}$ for some $b_i,b_j \in \mathbb R$, we obtain that $g$ is continuous as well.
\end{proof}

The following lemma is the generalization of \cite[Propositions 8.1.1 and 8.1.2]{Cignoli-Ottaviano-Mundici:AlgebraicFoundations}.

\begin{lemma}\label{l:VarietyInfiteGroups}
    \begin{enumerate}
        \item Let $\A_a$ be a pointed $\ell$-subgroup of $\R_a$ and $A$ be a dense subset of $\mathbb{R}$. Then $\HSP(\R_a)=\HSP(\A_a)$.
        


         

        \item Let $I$ be an infinite set of positive integers, $b > 0$ and $\A_a$ be a pointed $\ell$-subgroup of $\R_a$. Then $\HSP(\A_a \ltimes \R_b) = \HSP(\{\A_a \ltimes \Z_n \mid n \in I\})$.

        \item Let $I$ be an infinite set of negative integers, $b < 0$ and $\A_a$ be a pointed $\ell$-subgroup of $\R_a$. Then $\HSP(\A_a \ltimes \R_b) = \HSP(\{\A_a \ltimes \Z_n \mid n \in I\})$. 
        
    \end{enumerate}
\end{lemma}


\begin{proof}

    \begin{enumerate}
        \item Clearly, $\A_a \in \HSP(\R_a)$.
    
    For the other inclusion assume that some equation $\alpha(\vec{x})=0$ does not hold in $\R_a$. 
    Thus there is $\vec{r} \in \mathbb{R}^m$, where $m$ is the number of variables in $\alpha$, such that $\alpha(\vec {r}) \neq 0$. By Lemma~\ref{l:continuous_functions} the function $\vec{x} \mapsto \alpha(\vec{x})$ is continuous. Since the set $\{ x \in \mathbb{R}\ \mid x \neq 0\}$ is open in $\mathbb{R}$, the set $O:=\{\vec{x} \in \mathbb{R}^m \mid \alpha(\vec x) \neq 0\}$ is open in $\mathbb{R}^m$. We know that $\vec{r} \in O$ thus the set $O$ is nonempty.
    Since $A$ is a dense subset of $\mathbb{R}$, by \cite[Theorem 19.5]{Munkres:Topology} $A^m$ is a dense subset of $\mathbb{R}^m$ and therefore $A^m \cap O \neq \emptyset$.
    Thus there exists $\vec{a} \in A^m \cap O$ such that $\alpha(\vec{a}) \neq 0$. This shows that the equation $\alpha(\vec{x})=0$ is not valid in $\A_a$. Since $\alpha(\vec{x})=0$ was an arbitrary equation we derive that $\R_a \in \HSP(\A_a)$. Thus $\HSP(\A_a)=\HSP(\R_a)$.

    \item  Without loss of generality we can assume $b=1$, since $\R_1 \cong \R_b$ for any $b >0$ and by Lemma~\ref{l:iso-lex} $\A_a \ltimes \R_1 \cong \A_a \ltimes \R_b$ for any such $b$.
    Since $\Z_n \in \IS(\R_1)$ for each $n \in I$, by Lemma~\ref{l:iso-lex} we get $\A_a \ltimes \Z_n \in \IS(\A_a \ltimes \R_1)$ for each $n \in I$.

    For the other inclusion assume  that some equation $\alpha(\tuple{\vec{y},\vec{x}})=0$ does not hold in $\A_a \ltimes \R_1$. Let $m$ be the number of variables in $\alpha$. There exist $\tuple{\vec{s},\vec{r}} \in A^m \times \mathbb{R}^m$ such that $\alpha(\tuple{\vec{s},\vec{r}}) \neq 0$ in $\A_a \ltimes \R_1$.
    By Lemma~\ref{l:continuous_functions} the function $\vec{x} \mapsto \alpha(\vec s,\vec x)$ is continuous and thus the set $O:=\{\vec{x} \in \R^m \mid \alpha(\tuple{\vec{s},\vec{x}}) \neq 0\}$ is open in $\R^m$ and contains $\vec{r}$. 
    Therefore, there is $0 <\epsilon \in \mathbb R$ such that $V:=\{\vec{x} \mid |\vec{x}-\vec{r}| < m \cdot \epsilon\} \subseteq O$.
    We fix $n \in I$ such that $n > \frac{\sqrt{m}}{2 \epsilon}$. Let us denote $\B$ the pointed $\ell$-subgroup of $\R$ generated by $\frac{1}{n}$. We have $\epsilon > \frac{\sqrt{m}}{2 n}$ and by  \cite[Section 5, Chapter 4]{Conway-Sloane:PackingsLatticesGroups} it follows that $\B_1^m \cap V \neq \emptyset$. Since $V \subseteq O$ the equation $\alpha(\vec{s},\vec{x})=0$ is not valid in $\A_a \ltimes \B_1$.
    Since $\B_1 \cong \Z_n$ by Lemma~\ref{l:iso-lex} we get $\A_a \ltimes \B_1 \cong \A_a \ltimes \Z_n$, which proves $\A_a \ltimes \R_1 \in \HSP(\{\A_a \ltimes \Z_n \mid n \in I\})$.

    \item This can be proved similarly as the previous point. \qedhere \end{enumerate}
\end{proof}

\begin{lemma} \label{l:small_iso}
Let $a,k \in \mathbb{Z}$. Then $\Z_a \ltimes \Z_k \cong \Z_a \ltimes \Z_{k+a}$.
\end{lemma}

\begin{proof}
    Consider the mapping $\varphi:\Z_a \ltimes \Z_k \rightarrow \Z_a \ltimes \Z_{k+a}$ defined as $\tuple{x,y} \mapsto \tuple{x,y+x}$. We show $\varphi$ is an isomorphism of $\ell$-groups.
    Clearly, this is a group isomorphism with inverse mapping $\varphi^{-1}$ defined as $\tuple{x,y} \mapsto \tuple{x,y-x}$.
    It remains to check $\varphi$ and $\varphi^{-1}$ are order preserving. Assume $\tuple{a_1,b_1} \leq \tuple{a_2,b_2}$ for some $a_1,a_2,b_1,b_2 \in \mathbb{Z}$.
    \begin{enumerate}
        \item  If $a_1<a_2$ we have $$\varphi(\tuple{a_1,b_1})=\tuple{a_1,a_1+b_1}<\tuple{a_2,a_2+b_2}=\varphi(\tuple{a_2,b_2}).$$
        \item If $a_1=a_2$ we obtain $b_1 \leq b_2$ and $a_1+b_1 \leq a_2+b_2$. Therefore, 
    $$\varphi(\tuple{a_1,b_1})=\tuple{a_1,a_1+b_1}\leq \tuple{a_1,a_2+b_2}=\tuple{a_2,a_2+b_2}=\varphi(\tuple{a_2,b_2}).$$ 
    \end{enumerate}

    This shows that $\varphi$ preserves the ordering. In a similar way one can show that $\varphi^{-1}$ preserves the ordering as well. Thus $\varphi$ is indeed an isomorphism.
\end{proof}

\begin{defn} \label{d:radical}
Let $\A_b$ be a pointed Abelian $\ell$-group with $b \neq 0$. A convex $\ell$-subgroup of $\A$ that is maximal with respect to not containing $b$ is called a \emph{value} of $b$. 
If $\A_b$ is totally ordered, the value of $b$ is unique and we denote it by $\val(\A_b)$. Furthermore, we say that $\A_b$ is p-simple if $\val(\A_b) = \{0\}$.
\end{defn}

For any pointed Abelian $\ell$-group $\A_b$, $\val(\A_b)$ is not a \emph{pointed} Abelian $\ell$-group.
To treat $\val(\A_b)$ itself as an algebra in $p\mathbb{AL}$, we must assign it a designated element. When $\A$ is finitely generated, Lemma~\ref{l:decomposition_of_fin_gen_lgroups} applies to establish $\A \cong \A/\val(\A_b) \ltimes \val(\A_b)$. Under this isomorphism, the original designated element $b \in A$ must map to a specific pair. The first coordinate is its projection $b + \val(\A_b)$ in the quotient, while the second coordinate is a uniquely determined element $c \in \val(\A_b)$.
We use this point $c$ as the designated element for $\val(\A_b)$. Henceforth, we adopt the convention that whenever $\val(\A_b)$ is treated as a pointed Abelian $\ell$-group, its designated element $\fa^{\val(\A_b)}$ is precisely this element $c$, yielding:

\[ \A_b \cong (\A/\val(\A_b))_{b + \val(\A_b)} \ltimes (\val(\A_b))_c. \]

Let $\A_b$ be a totally ordered pointed Abelian $\ell$-group with $b \neq 0$ and let $\B_b$ be the totally ordered strongly pointed convex $\ell$-subgroup of $\A_b$. It can be shown that $\val(\A_b) = \val(\B_b)$. Since $\B_b$ is strongly pointed, $\val(\B_b)$ is a maximal convex $\ell$-subgroup of $\B$ and hence the quotient $\B/\val(\B_b)$ must be a simple Abelian $\ell$-group. Therefore, according to Hölder's Theorem (see Theorem~\ref{t:Holder}), $\B_b/\val(\B_b)$ is isomorphic to a pointed $\ell$-subgroup of the real numbers. 
Since any pointed Abelian $\ell$-group $\A_b$ with $b \neq 0$ has a unique convex strongly pointed $\ell$-subgroup, we can state the following definition.

\begin{defn}
Let $\A_b$ be a totally ordered pointed Abelian $\ell$-group with $b \neq 0$ and $\B_b$ be its strongly pointed convex $\ell$-subgroup.
We define the rank of $\A_b$ as 
$$\rank(\A_b)=
\begin{cases}
    n & \text{if } \B_b/\val(\B_b) \cong \Z_n, \\
    \infty & \text{if } \B_b/\val(\B_b) \ncong \Z_n \text{ for any } n \in \mathbb Z\; \& \;b>0, \\
    - \infty & \text{if } \B_b/\val(\B_b) \ncong \Z_n \text{ for any } n \in \mathbb Z\; \& \; b<0. \\
\end{cases}
$$

Moreover, for any totally ordered $0$-pointed Abelian $\ell$-group $\A_0$ we set  \newline
$\rank(\A_0)=0$.

\end{defn}

\begin{lemma} \label{l:simple rank}
    Let $\A_a$ be a totally ordered strongly pointed Abelian $\ell$-group and $\rank(\A_a)=n \in \mathbb Z \setminus\{0\}$. Moreover, assume $\A_a$ is p-simple. Then $\A_a \cong \Z_n$.
\end{lemma}

\begin{proof}
    Follows from the definitions of rank and p-simple Abelian $\ell$-group.
\end{proof}

\begin{lemma} \label{l:Non-SimpleGroupsRank}
Let $\A_a$ be a finitely generated totally ordered pointed Abelian $\ell$-group which is not p-simple and has rank $n \in \mathbb Z \setminus\{0\}$. Then $\HSP(\A_a)=\HSP(\Z_n \ltimes \Z_0)$.

\end{lemma}

\begin{proof}
    By Theorem~\ref{t:strong_unit} we can assume without loss of generality that $\A_a$ is strongly pointed. Since $\rank(\A_a)=n$, by Lemma~\ref{l:DecomPointLgrp} and Lemma~\ref{l:iso-lex} we obtain $\A_a /\val(\A_a) \cong \Z_n$ and thus
    $$\A_a \cong \A_a /\val(\A_a) \ltimes \val(\A_a) \cong \Z_n \ltimes \val(\A_a).$$ We first show $$\Z_n \ltimes \Z_0 \in \HSP(\Z_n \ltimes \val(\A_a)).$$

    Since $\A_a$ is finitely generated and totally ordered, $\val(\A_a)$ has a strong unit and thus 
    Lemma~\ref{l:lex1} gives us $$\Z_0 \ltimes \val(\A_a) \in \ISPU(\val(\A_a))$$ and thus by Lemma~\ref{l:lexproducts_partembeddings} we obtain $$\Z_n \ltimes \Z_0 \ltimes \val(\A_a) \in \ISPU(\Z_n \ltimes \val(\A_a)).$$ 
    Finally, we observe $$\Z_n \ltimes \Z_0 \in \H(\Z_n \ltimes \Z_0 \ltimes \val(\A_a))$$ and thus we have $$\Z_n \ltimes \Z_0 \in \HSP(\Z_n \ltimes \val(\A_a)).$$

    It remains to show the other inclusion, i.e.\ that 
    $$\Z_n \ltimes \val(\A_a) \in \HSP(\Z_n \ltimes \Z_0).$$ By Lemma~\ref{l:small_iso} we have 
    $$\Z_n \ltimes \Z_{kn} \in \I(\Z_n \ltimes \Z_0) \text{ for each }k \in \mathbb{Z}.$$ Thus by Lemma~\ref{l:VarietyInfiteGroups} we obtain $$\Z_n \ltimes \R_{1}, \Z_n \ltimes \R_{-1} \in \HSP(\Z_n \ltimes \Z_0).$$ Also, by Lemma~\ref{l:about R1} we have $\R_0 \in \ISPU(\Z_0)$ and thus by Lemma~\ref{l:lexproducts_partembeddings} and Lemma~\ref{l:iso-lex} we have 
    $$\Z_n \ltimes \R_0 \in \ISPU(\Z_n \ltimes \Z_0).$$

    Since $\val(\A_a)$ is totally ordered, by Lemma~\ref{l:about R1} and \cite[Theorem 5.6]{Bergman:UniversalAlgebra} we have
    $$\val(\A_a) \in \ISPU(\R_{-1},\R_0,\R_1)=\ISPU(\R_{-1}) \cup \ISPU(\R_{0}) \cup \ISPU(\R_{1}).$$ 
    By Lemma~\ref{l:lexproducts_partembeddings} we obtain 
    $$\Z_n \ltimes \val(\A_a) \in \ISPU(\Z_n \ltimes \R_{-1}) \cup \ISPU(\Z_n \ltimes \R_{0}) \cup \ISPU(\Z_n \ltimes \R_{1}).$$
     Since $$\Z_n \ltimes \R_{-1}, \Z_n \ltimes \R_{0}, \Z_n \ltimes \R_{1} \in \HSP(\Z_n \ltimes \Z_0),$$ we obtain $$\Z_n \ltimes \val(\A_a) \in \HSP(\Z_n \ltimes \Z_0).$$
    This proves the claim.
\end{proof}

\begin{lemma}
    Let $\A_a$ be a finitely generated totally ordered pointed Abelian $\ell$-group.

    \begin{enumerate}
        \item If $\rank(\A_a)=\infty$ then $\HSP(\A_a)=\HSP(\R_{1})$.
        \item If $\rank(\A_a)=-\infty$ then $\HSP(\A_a)=\HSP(\R_{-1})$.
    \end{enumerate}
\end{lemma}

\begin{proof}
    By Theorem~\ref{t:strong_unit} we can assume without loss of generality that $\A_a$ is strongly pointed. Let us assume $\rank(\A_a)=\infty$.
    Since $\val(\A_a)$ is a maximal convex $\ell$-subgroup of $\A_a$, the Abelian $\ell$-group $\A_a/\val(\A_a)$ must be p-simple and thus by Theorem~\ref{t:Holder} $\A_a/\val(\A_a) \in \IS(\R_b)$ for some $0<b \in \mathbb{R}$. Let us without loss of generality assume $\A_a/\val(\A_a) \in \S(\R_1)$. 
    By Lemma~\ref{l:about rationals} we obtain $\A_a/\val(\A_a) \cong \Z_n$ for some $n \in \mathbb{Z}$ or the universe of $\A_a/\val(\A_a)$ is a dense subset of $\mathbb{R}$. Since $\rank(\A_a)=\infty$, the universe of $\A_a/\val(\A_a)$ is dense in $\mathbb{R}$. Thus by Lemma~\ref{l:VarietyInfiteGroups} we get $\HSP(\A_a/\val(\A_a))=\HSP(\R_1)$.
    Since $\A_a/\val(\A_a) \in \H(\A_a)$ and by Lemma~\ref{l:about R2} $\A_a \in \HSP(\R_1)$ we obtain $\HSP(\R_1)=\HSP(\A_a)$.

    The case $\rank(\A_a)=-\infty$ is analogous.
\end{proof}

Altogether, we obtained the following characterization.

\begin{thm} \label{t:join-irreducible varieties}
    Let $\A_a$ be a finitely generated non-trivial totally ordered pointed Abelian $\ell$-group. Then the following holds.

    \begin{enumerate}

        \item $\HSP(\A_a)=\HSP(\Z_0)=\HSP(\R_0)$ iff $a=0$.
        \item  $\HSP(\A_a)=\HSP(\Z_n)$ iff $\A_a$ is p-simple and $\rank(\A_a)=n$.
        \item $\HSP(\A_a)=\HSP(\Z_n \ltimes \Z_0)$ iff $\A_a$ is not p-simple and $\rank(\A_a)=n$.
        \item $\HSP(\A_a)=\HSP(\R_1)$ iff $\rank(\A_a)=\infty$.
        \item $\HSP(\A_a)=\HSP(\R_{-1})$ iff $\rank(\A_a)=-\infty$.
    \end{enumerate}
\end{thm}



   

\section{Axiomatization of subvarieties of \texorpdfstring{$p\mathbb{AL}$}{pAL}} \label{s:Axiomatization}

In this section we focus on axiomatizing all the subvarieties of $p\mathbb{AL}^+$. To achieve this we define the following equations: \rul{\text{s-rank}_n},\rul{\text{rank}_n}, \rul{\text{div}_{p,n}} and \rul{\text{mix}_{p,n}}. Observe that all these four equations are valid for all parameters in all negatively pointed Abelian $\ell$-groups. Consequently, we introduce another four dual equations \rul{\text{s-rank}^d_n},\rul{\text{rank}^d_n}, \rul{\text{div}^d_{p,n}} and \rul{\text{mix}^d_{p,n}}, which we use to characterize all subvarieties of $p\mathbb{AL}^-$. Finally, we combine these strategies to provide a characterization of all subvarieties of $p\mathbb{AL}$.


\begin{lemma} \label{l:s-rank}
    Let $\A_a$ be a finitely generated totally ordered pointed Abelian $\ell$-group and $n \geq 0$.
    Let us consider the following equation.
    
    \begin{equation}\tag{s-rank$_n$}
         (n \cdot x -\fa) \lor (-x) \geq 0.
    \end{equation}
    This equation is satisfied in $\A_a$ iff $a \leq 0$ or $\A_a$ is p-simple with $\rank(\A_a) \leq n$.
    
\end{lemma}

\begin{proof}
First, if $a=0$, the equation is satisfied by Corollary \ref{c:minimal variety}. Now, we can assume $a \neq 0$.
By Theorem~\ref{t:strong_unit} we can assume $\A_a$ is strongly pointed. Thus by Lemma~\ref{l:simple rank} we need to check that $$\A_a \vDash \rul{\text{s-rank}_n} \text{ iff } a \leq 0 \text{ or } \A \cong \Z_m \text{ for some }m \leq n.$$

Let $e$ be a fixed evaluation on $\A_a$. We have $\A_a \vDash_e \rul{\text{s-rank}_n}$ iff  $\A_a \vDash_e x \leq 0$ or $\A_a \vDash_e n \cdot x \geq a$.
Therefore $\A_a \vDash \rul{\text{s-rank}_n}$ iff $\A_a \vDash_e n \cdot x \geq a$ for all evaluations $e$, such that $e(x)>0$.

We need to distinguish several cases.

\begin{enumerate}
    \item For $a \leq 0$ and $e(x)>0$ we have $n \cdot e(x) \geq 0 \geq a$. Thus $\A_a \vDash \rul{\text{s-rank}_n}$ for $a \leq 0$.

    \item For $\A_a=\Z_m$, $0<m \leq n$ and $0<e(x)$ we have  $n \cdot e(x) \geq n \cdot 1 \geq m$. Therefore $\Z_m \vDash \rul{\text{s-rank}_n}$ for $m \leq n$.

    \item  For $\A_a=\Z_m$, $0 \leq n < m$ set $e(x)=1$. We have $n \cdot e(x) = n < m$, thus $\Z_m \nvDash \rul{\text{s-rank}_n} $ for $m>n$.

    \item For $a>0$ and $\A_a$ not p-simple we have $\val(\A_a) \neq 0$ and by Lemma~\ref{l:DecomPointLgrp} there is an isomorphism $\iota:\A_a/\val(\A_a) \ltimes \val(\A_a) \rightarrow \A_a$. Pick $0 < v \in \val(\A_a)$ and we consider an evaluation $e$ on $\A$, where $e(x)=\iota(\tuple{0,v})$.
    Since $n \cdot v \in \val(\A_a)$, we have $n \cdot e(x) = \iota(\tuple{0,n \cdot v}) < a$, thus showing $\A_a \nvDash \rul{\text{s-rank}_n} $ for not p-simple $\A_a$ with $a>0$. \qedhere
\end{enumerate}
\end{proof}

Let us note that \rul{\text{s-rank}_0} is equivalent to the equation $f \leq 0$.
For $m \in \mathbb Z \setminus \{0\}$, we let $\div(m)$ denote the set of all divisors of $m$. For $m=0$, we set $\div(m)=\N$.



\begin{lemma} \label{l:divisibity equation}
Let $m \in \mathbb{Z}$, $n,p \in \mathbb{N}$ and $m \leq n$.

Let us consider the following equation.

    \begin{equation} \tag{div$_{p,n}$}
        ((n+1) \cdot ((p\cdot x - \fa) \lor (\fa - p \cdot x))-\fa) \lor  -x \geq 0
    \end{equation}

The following conditions are equivalent.
      \begin{enumerate}
        \item $\Z_m \vDash \rul{div_{p,n}}$,
        \item $\Z_m \ltimes \Z_0 \vDash \rul{div_{p,n}}$,
        \item  $m \leq 0$ or $p \notin \div(m)$.
    \end{enumerate}
\end{lemma}
\begin{proof}
    Since $\Z_m \in \H(\Z_m \ltimes \Z_0)$, it is enough to check that $\Z_m \ltimes \Z_0 \vDash \rul{div_{p,n}}$ if $m \leq 0$ or $p \notin \div(m)$ and that $\Z_m \nvDash \rul{div_{p,n}}$ if $0<m$ and $p \in \div(m)$.

    \begin{enumerate}[leftmargin=*]
        \item Let $e$ be a fixed evaluation on $\Z_m \ltimes \Z_0$. 
        We have 
        $$\Z_m \ltimes \Z_0 \vDash_e \rul{div_{p,n}} \text{ iff } e(x) \leq 0 \text{ or }(n+1) \cdot |p \cdot e(x)-\tuple{m,0}| \geq \tuple{m,0}.$$
    Therefore $\Z_m \ltimes \Z_0 \vDash \rul{\text{div}_{p,n}}$ iff $(n+1) \cdot |p \cdot e(x)-\tuple{m,0}| \geq \tuple{m,0}$ for all evaluations $e$, such that $e(x)>0$.  
    For an evaluation $e(x)>0$ and $m \leq 0$, we get $$(n+1) \cdot |p \cdot e(x)-\tuple{m,0}| \geq \tuple{0,0} \geq \tuple{m,0}.$$ This tells us that $\Z_m \ltimes \Z_0 \vDash \rul{\text{div}_{p,n}}$ for $m \leq 0$.

    Now assume $0<m \leq n, p \notin \div(m)$, $e(x)=\tuple{a_1,a_2}$ and $a_1 \geq 0$. Since $p \notin \div(m)$ we get $p \cdot a_1 \neq m$, and thus $|p \cdot a_1-m| \geq 1$. Consequently, $$(n+1) \cdot |p \cdot a_1-m| > m \cdot |p \cdot a_1-m|\geq m.$$ Thus
    \begin{align*}
        (n+1) \cdot |p \cdot \tuple{a_1,a_2}-\tuple{m,0}|  &=  |\tuple{(n+1) \cdot (p \cdot a_1-m),((n+1) \cdot p \cdot a_2)}| \\ &\geq \tuple{m,0}.
    \end{align*}
    Therefore, $\Z_m \ltimes \Z_0 \vDash \rul{div_{p,n}}$ for $m>0$ such that $p \notin \div(m)$.

    \item Now assume $0<m$ and $p \mid m$. Since $p \mid m$, there is $r \in \mathbb{N}$ such that $r \cdot p=m$.
    Let us consider an evaluation $e$ on $\Z_m$, where $e(x)=r$. We have 
     $$(n+1) \cdot |p \cdot r-m|=(n+1) \cdot 0=0 <m.$$

     This shows $\Z_m \nvDash \rul{div_{p,n}}$ if $0<m$ and $p \in \div(m)$.

    \end{enumerate}
      
This completes the proof.
\end{proof}

Now we have enough tools to axiomatize the variety $\HSP(\Z_n)$.

\begin{thm} \label{t:axiomatization Zn}
    Let $n \in \N$. The variety $\HSP(\Z_n)$ can be axiomatized as a subvariety of $p\mathbb{AL}^+$ by the following set of formulas:
    $\{\rul{\text{s-rank}_{n}}\} \cup \{\rul{\text{div}_{p,{n}}} \mid p \notin \div(n), p < n\}$.
\end{thm}

\begin{proof}
    By Lemma~\ref{l:s-rank} we have $\Z_n \vDash \rul{\text{s-rank}_{n}}$ and by Lemma~\ref{l:divisibity equation} we have $\Z_n \vDash \rul{\text{div}_{p,{n}}}$ for all $p \notin \div(n)$.

    Now assume there is a finitely generated $\A_a \in p\mathbb{AL}^+_{\FSI}$, such that $\A_a \vDash \rul{\text{s-rank}_{n}}$ and $\A_a \vDash \rul{\text{div}_{p,{n}}}$ for all $p \notin \div(n)$. We will show $\A_a \in \HSP(\Z_n)$.

    By Theorem~\ref{t:Clifford}, $\A_a$ is totally ordered. If $a=0$ we obtain the desired result by applying Corollary \ref{c:minimal variety}. Let us now assume that $a \neq 0$. By Theorem~\ref{t:strong_unit} we can assume without loss of generality that $\A_a$ is strongly pointed. By Lemma~\ref{l:s-rank} we obtain $\A_a$ must be p-simple with rank less or equal to $n$ and thus by Lemma~\ref{l:simple rank} $\A_a \cong \Z_k$ for some $k \leq n$.
    By Lemma~\ref{l:divisibity equation} we obtain that $k$ divides $n$. That means $\A_a \cong \Z_k \in \IS(\Z_n)$ and thus $\A_a \in \HSP(\Z_n)$.
\end{proof}

Since Theorem~\ref{t:final positive} also provides an axiomatization for the variety generated by $\Z_n$, it might seem that this would render Theorem~\ref{t:axiomatization Zn} redundant. However, this is not the case. The varieties generated by $\Z_n$ are structurally simpler than the general case, as they are not generated by any $\ell$-groups of the form $\Z_m \ltimes \Z_0$. Consequently, Theorem~\ref{t:axiomatization Zn} provides a much simpler and more direct axiomatization than the one required by the general framework of Theorem~\ref{t:final positive}. This is particularly appealing because the varieties generated by $\Z_n$ correspond via the Mundici functor to the varieties generated by a single finite simple MV-algebra, which are among the most fundamental structures in the theory of MV-algebras.


\begin{lemma} \label{l:rank equation}
    Let $n \geq 0$ and $m \in \mathbb{Z}$. 
    Let us consider the following equation.

    \begin{equation} \tag{rank$_n$}
        ((2n+1) \cdot x - 2 \cdot \fa) \lor (\fa-(2n+2) \cdot x) \lor -x \geq 0.
    \end{equation}

    The following conditions are equivalent.
    \begin{enumerate}
        \item $\Z_m \vDash \rul{\text{rank}_n}$,
        \item $\Z_m \ltimes \Z_0 \vDash \rul{\text{rank}_n}$,
        \item $m \leq n$.
    \end{enumerate}
    
\end{lemma}

\begin{proof}
    Since $\Z_n \in \H(\Z_n \ltimes \Z_0)$, we only have to show  $\Z_m \nvDash \rul{\text{rank}_n}$ for $m>n$ and $\Z_m \ltimes \Z_0 \vDash \rul{\text{rank}_n}$ for $m \leq n$.

    \begin{enumerate}
        \item First we show $\Z_m \nvDash \rul{\text{rank}_n}$ for $n <m$. Since $\Z_m$ is Archimedean, there is a maximal $k \geq 1$ such that $k \cdot (2n+1)<2m$. Let us note that from the maximality of $k$ it follows $2k \cdot (2n+1) \geq 2m$ and thus $k \cdot (2n+1) \geq m$. Let $e$ be an evaluation on $\Z_m$ such that $e(x)=k$.
    Now we have $$(2n+1)\cdot e(x)-2 \cdot e(\fa)=(2n+1)\cdot k-2 \cdot m < 0$$ by the definition of $k$. 
    Moreover, we have $$e(\fa)-(2n+2) \cdot e(x) = m-k(2n+2) \leq k (2n+1)-k(2n+2)=-k<0.$$ Clearly, also $-e(x)=-k<0$.
    This shows $\Z_m \nvDash_e \rul{\text{rank}_n}$ for $n<m$.

        \item It remains to show that $\Z_m \ltimes \Z_0 \vDash \rul{\text{rank}_n}$ for $m \leq n$.
    Trivially, we have $\Z_m \ltimes \Z_0 \vDash_e \rul{\text{rank}_n}$ for any evaluation $e(x) \leq 0$.
    Therefore, from now on we can focus only on evaluations such that $e(x)>0$.
    For $m \leq 0$ and evaluation $e$ such that $e(x)>0$ we have $(2n+1) \cdot e(x) -2e(\fa) \geq 0$. Thus $\Z_m \ltimes \Z_0 \vDash_e \rul{\text{rank}_n}$ for $m \leq 0$.
    
    Now assume $m>0$.
    First, let us consider the case $e(x)=\tuple{0,b}$ for some $b \in \mathbb{Z}$. We have $$e(\fa)-(2n+2) \cdot e(x) = \tuple{m,0}-(2n+2) \cdot \tuple{0,b} = \tuple{m,-(2n+2) \cdot b}>0.$$
    
    Finally we have to check the case when   $e(x)=\tuple{a,b}$ for some $a,b \in \mathbb{Z}$ and $a>0$.  We have   
    \begin{align*}
        (2n+1) \cdot e(x) - 2 \cdot e(\fa) &= (2n+1) \cdot \tuple{a,b} - 2 \cdot \tuple{m,0} \\ &\geq (2n+1) \cdot \tuple{1,b} - 2 \cdot \tuple{n,0}\\ &=\tuple{(2n+1) -2n,(2n+1) \cdot b} \\&=\tuple{1,(2n+1) \cdot b} \geq 0.
    \end{align*}
        
    This shows that indeed $\Z_m \ltimes \Z_0 \vDash \rul{\text{rank}_n}$ for $m \leq n$. \qedhere
 \end{enumerate}
\end{proof}

Now we can introduce the final equation, which we need to axiomatize the subvarieties of $p\mathbb{AL}$.
This equation is equivalent to the disjunction of the equations \rul{\text{s-rank}_n} and \rul{div_{p,n}}.

\begin{lemma} \label{l:mix equation}
    Let $\A$ be a finitely generated totally ordered pointed Abelian $\ell$-group and $0 \leq p \leq n$.
Let us consider the following equation.
    \begin{equation} \tag{mix$_{p,n}$}
        ((n+1) \cdot ((p\cdot x - \fa) \lor (\fa - p \cdot x))-\fa) \lor (-x) \lor (n \cdot y -\fa) \lor (-y) \geq 0
    \end{equation}
We have $\A \vDash \rul{mix_{p,n}}$ if and only if $\A \vDash \rul{s\text{-}rank_n}$ or $\A \vDash \rul{div_{p,n}}$.
    
\end{lemma}

\begin{proof}

Denote 
 \begin{equation}\tag{s-rank$_n (y/x)$}
         (n \cdot y -\fa) \lor (-y) \geq 0.
    \end{equation}

Since $\A$ is totally ordered we know that for each evaluation $e$ of $\A$ we have $\A \vDash_e  \rul{mix_{p,n}}$ iff $\A \vDash_e  \rul{div_{p,n}}$ or $\A \vDash_e  \rul{\text{s-rank}_n(y/x)}$. Therefore if $\A \vDash  \rul{div_{p,n}}$ or $\A \vDash  \rul{\text{s-rank}_n(y/x)}$ then clearly $\A \vDash  \rul{mix_{p,n}}$. For the other implication assume $\A \nvDash  \rul{div_{p,n}}$ and $\A \nvDash  \rul{\text{s-rank}_n(y/x)}$. Then there are evaluations $e_1,e_2$ such that $\A \nvDash_{e_1}  \rul{div_{p,n}}$ and  $\A \nvDash_{e_2}  \rul{\text{s-rank}_n(y/x)}$. We will define an evaluation $e_3$ on $\A$, such that $e_3(x)=e_1(x)$ and $e_3(y)=e_2(y)$. Then we get $\A \nvDash_{e_3}  \rul{mix_{p,n}}$.
Therefore we have $\A \vDash \rul{mix_{p,n}}$ iff $\A \vDash  \rul{div_{p,n}}$ or $\A \vDash  \rul{\text{s-rank}_n(y/x)}$. Since $\A \vDash  \rul{\text{s-rank}_n(y/x)}$ if and only if $\A \vDash  \rul{\text{s-rank}_n}$, we obtain that $\A \vDash \rul{mix_{p,n}}$ if and only if $\A \vDash  \rul{div_{p,n}}$ or $\A \vDash  \rul{\text{s-rank}_n}$, which completes the proof.
\end{proof}

\begin{cor} \label{c:mix equation}
Let $0 \leq p \leq n$ and $m \leq n$.
    We have $\Z_m \vDash \rul{mix_{p,n}}$. 
    Moreover, $\Z_m \ltimes \Z_0 \vDash \rul{mix_{p,n}}$ iff $p \notin \div(m)$ or $m \leq 0$.
\end{cor}

At this point we have all we need to provide the axiomatization of subvarieties of $p\mathbb{AL}$.
First we will discuss subvarieties of positively and negatively pointed Abelian $\ell$-groups.

\begin{thm}  \label{t:final positive}
    Any non-trivial proper subvariety of $p\mathbb{AL}^+$ is of the form $$\mathcal{V}_{I,J}=\HSP(\Z_i,\Z_j \ltimes \Z_0 \mid i \in I, j \in J)$$ for some finite sets $\{0\} \subseteq J \subseteq I \subsetneq \N$, where both $I \setminus \{0\}$ and $J  \setminus \{0\}$ are closed under divisors.

    Moreover, $\mathcal{V}_{I,J}$ is axiomatized by the following set of equations:
$$S_{I,J}=\{\rul{rank_n}\} \cup \{\rul{div_{p,n}} \mid p \notin I\} \cup \{\rul{mix_{p,n}} \mid p \in I \setminus J\}, $$

where $n=\max I$.
    
\end{thm}

\begin{proof}
    First we show that any proper subvariety of $p\mathbb{AL}^+$ is equal to $\mathcal{V}_{I,J}$ for some finite sets $\{0\} \subseteq J \subseteq I$.
    Let $\mathcal{V}$ be an arbitrary subvariety of $p\mathbb{AL}^+$. 
    Let us denote $I=\{i \mid \Z_i \in \mathcal{V}\}$ and $J=\{j \mid \Z_j \ltimes \Z_0 \in \mathcal{V}\}$. By Corollary \ref{c:minimal variety} we have $0 \in I \cap J$.   
    We show $\mathcal{V}=\mathcal{V}_{I,J}$. Since $\mathcal{V}$ is a \textit{proper} subvariety of $p\mathbb{AL}^+$, by Lemma~\ref{l:about R2} we know that $\R_1 \notin \mathcal{V}$. By Theorem~\ref{t:join-irreducible varieties} we have 
    \begin{align*}
        \mathcal{V}&=\HSP(\{\A \in \mathcal{V}_{\FSI} \mid \A \text{ is fin.\ gen.}\})\\ &=\HSP(\{\Z_i,\Z_j \ltimes \Z_0 \mid i \in I, j \in J\})=\mathcal{V}_{I,J}.
    \end{align*}
    Clearly, both $I  \setminus \{0\}$ and $J \setminus \{0\}$ are closed under divisors, since for any $d,k \in \N \setminus \{0\}$ such that $d \in \div(k)$ we have $\Z_d \in \IS(\Z_k)$ and  $\Z_d \ltimes \Z_0 \in \IS(\Z_k \ltimes \Z_0)$. 
    Since for each $j \in J$ we have $\Z_j \in \H(\Z_j \ltimes \Z_0)$, we get $J \subseteq I$. The set $I$ (and consequently $J$) must be finite, otherwise we would get by Lemma~\ref{l:VarietyInfiteGroups} that $\R_1 \in \mathcal{V}$.
    This proves $\mathcal{V}=\mathcal{V}_{I,J}$ for some finite sets $J \subseteq I$.

    Now we have to show the axiomatization of $\mathcal{V}_{I,J}$.
    First, we will argue that all the equations from $S_{I,J}$ hold in $\mathcal{V}_{I,J}$. Since $n=\max I \geq \max J$, by Lemma~\ref{l:rank equation} we know that $\Z_i \vDash \rul{rank_n}$ for all $i \in I$ and $\Z_j \ltimes \Z_0 \vDash \rul{rank_n}$ for all $j \in J$. Thus $\mathcal{V}_{I,J} \vDash \rul{rank_n}$.
    
    Let us fix $p \notin I$. Since $I \setminus \{0\}$ is closed under divisors, we get that $p \notin \div(i)$ for all $i \in I \setminus \{0\}$.
    Using $J \subseteq I$ and  by Lemma~\ref{l:divisibity equation} we obtain that
    $\Z_i \vDash \rul{div_{p,n}}$ for each $i \in I$ and $\Z_j \ltimes \Z_0 \vDash \rul{div_{p,n}}$ for each $j \in J$. Therefore, $\mathcal{V}_{I,J} \vDash \rul{div_{p,n}}$ for each $p \notin I$.


    Now let us fix $p \in I \setminus J$. Since $p \notin J$ by the same argument as in the previous paragraph we obtain $\Z_j \ltimes \Z_0 \vDash \rul{div_{p,n}}$ for each $j \in J$. By Lemma~\ref{l:s-rank} we have $\Z_i \vDash \rul{\text{s-rank}_n}$ for each $i \in I$ and thus by Lemma~\ref{l:mix equation} we get $\Z_i \vDash \rul{mix_{p,n}}$ and $\Z_j \ltimes \Z_0 \vDash \rul{mix_{p,n}}$ for each $i \in I$ and $j \in J$.
    Thus $\mathcal{V}_{I,J} \vDash \rul{mix_{p,n}}$ for each $p \in I \setminus J$.
    
    This concludes that $\mathcal{V}_{I,J} \vDash \rul{eq}$ for each $\rul{eq} \in S_{I,J}$.    
    It remains to show $\mathcal{V}_{I,J}$ is uniquely determined by the equations from $S_{I,J}$.

    Let $\A_a \in p\mathbb{AL}^+_{\FSI}$ be a finitely generated satisfying all the equations from $S_{I,J}$. We will show $\A_a \in \mathcal{V}_{I,J}$. First, let us note that clearly $\HSP(\A_a) \neq \HSP(\R_1)$ since from Lemma~\ref{l:about R2} we know $\HSP(\R_1)=p\mathbb{AL}^+$ and $\A$ satisfies equations from $S_{I,J}$, which are not generally valid in $p\mathbb{AL}^+$.
    By Corollary \ref{c:minimal variety} and Theorems \ref{t:strong_unit} and \ref{t:join-irreducible varieties} it is enough to consider the cases $\A_a \in \{\Z_i,\Z_j \ltimes \Z_0\mid i,j \in \N\}$.

    For $\A_a=\Z_m$ we get $m \leq n$, by applying Lemma~\ref{l:rank equation} using $\Z_m \vDash \rul{rank_n}$. Since $\Z_m \vDash \rul{div_{p,n}}$ for all $p \notin I$, we get by Lemma~\ref{l:divisibity equation} that $m \in I$.

    For $\A_a=\Z_m \ltimes \Z_0$ we similarly get $m \leq n$ and $m \in I$. Since for every $p \in I \setminus J$ we have $\Z_m \ltimes \Z_0 \vDash \rul{mix_{p,n}}$, by Corollary~\ref{c:mix equation} we get $m \in J$.

    Thus $\A_a \in \mathcal{V}_{I,J}$, which completes the proof.
\end{proof}

The structural part of Theorem~\ref{t:final positive} was proved in \cite{Young:Varieties_of_pointed_Abelian_l-groups}. This result can be obtained by applying the Mundici functor, Corollary \ref{c:minimal variety}, Theorem~\ref{t:strong_unit}, and Komori classification, an approach we will discuss later in Section \ref{s:Mundici} and demonstrate in Theorems \ref{t:Komori} and \ref{t:Gispert-groups}.

Despite the fact that the statement of Theorem~\ref{t:final positive} was about proper varieties, we can also claim that any variety of $p\mathbb{AL}^+$ is equal to $\mathcal{V}_{I,J}$ for some sets $J \subseteq I$, since $p\mathbb{AL}^+=\mathcal{V}_{\N,\N}$. To formally extend our axiomatization to this case, we define the corresponding set of equations as $S_{\N,\N} = \emptyset$.

In this chapter, we have so far been describing the subvarieties of $p\mathbb{AL}^+$. However, we could state all the results similarly for $p\mathbb{AL}^-$ with accordingly modified equations:

		\begin{equation}\tag{s-rank$_n^d$}
			(\fa-n \cdot x) \lor x \geq 0.
		\end{equation}
		
		\begin{equation} \tag{rank$_n^d$}
			(2 \cdot \fa-(2n+1) \cdot x) \lor ((2n+2) \cdot x-\fa) \lor x \geq 0.
		\end{equation}
		
		\begin{equation} \tag{div$_{p,n}^d$}
			(\fa -(n+1) \cdot ((p\cdot x - \fa) \land (\fa - p \cdot x))) \lor x \geq 0.
		\end{equation}
		
		\begin{equation}\tag{mix$_{p,n}^d$}
				(\fa-n \cdot y) \lor y	\lor (\fa-((n+1) \cdot ((p\cdot x - \fa) \land (\fa - p \cdot x)))) \lor x \geq 0
		\end{equation}

Using these equations we can state a dual result to Theorem $\ref{t:final positive}$.

\begin{thm} \label{t:final negative}
    Any non-trivial proper subvariety of $p\mathbb{AL}^-$ is of the form $$\mathcal{V}^d_{I,J}=\HSP(\Z_{-i},\Z_{-j} \ltimes \Z_0 \mid i \in I, j \in J)$$ for some finite sets $\{0\} \subseteq J \subseteq I  \subsetneq \N$, where both $I \setminus \{0\}$ and $J  \setminus \{0\}$ are closed under divisors.
    
    Moreover, $\mathcal{V}^d_{I,J}$ is axiomatized by the following set $S_{I,J}^d$ of equations:
$$S_{I,J}^d=\{\rul{rank_n^d}\} \cup \{\rul{div_{p,n}^d} \mid p \notin I\} \cup \{\rul{mix_{p,n}^d} \mid p \in I \setminus J\}, $$

where $n=\max I$.
\end{thm}

Now we can combine Theorems \ref{t:final negative} and \ref{t:final positive} into one theorem covering all subvarieties of $p\mathbb{AL}$.

\begin{thm} \label{t:final final}
    Any non-trivial subvariety of $p\mathbb{AL}$ can be written as $\mathcal{V}=\mathcal{V}_{I_1,J_1} \lor \mathcal{V}^d_{I_2,J_2}$ for some (not necessarily finite) sets $I_1,I_2,J_1,J_2$.
    Such a variety is axiomatized by the equations from $S_{I_1,J_1} \cup S_{I_2,J_2}^d$, where $S_{I_1,J_1}$ and $S^d_{I_2,J_2}$ are defined in Theorems \ref{t:final positive} and \ref{t:final negative}, setting $S_{\N,\N}=S^d_{\N,\N}=\emptyset$.
\end{thm}

\begin{figure}[ht!]
\centering
\begin{tikzpicture}[
    scale=0.64, transform shape,
    znode/.style={text=blue!80!black},
    knode/.style={text=red!80!black},
    arrow/.style={->, thick, gray}
]

\node (GlobalTop) at (0, 9) {$p\mathbb{AL}$};
\node[znode] (Z0) at (0, -1) {$\mathcal{V}(\mathbf{Z}_0)$};
\node (Triv) at (0, -2) {$\mathcal{T}$};
\draw (Triv) -- (Z0);

\begin{scope}[xshift=-4.5cm]
    \node (Top_p) at (-2, 8) {$p\mathbb{AL}^+$};
    \node[znode] (Z2) at (2, 1.5) {$\mathcal{V}(\mathbf{Z}_2)$};
    \node[znode] (Z3) at (-0.5, 1.5) {$\mathcal{V}(\mathbf{Z}_3)$};
    \node[znode] (Z5) at (-3, 1.5) {$\mathcal{V}(\mathbf{Z}_5)$};
    \node[znode] (Z4) at (3.5, 2.75) {$\mathcal{V}(\mathbf{Z}_4)$};
    \node[znode] (Z6) at (0.75, 2.75) {$\mathcal{V}(\mathbf{Z}_6)$};
    \node[znode] (Z9) at (-1.5, 2.75) {$\mathcal{V}(\mathbf{Z}_9)$};
    \node[znode] (Z12) at (1.5, 4) {$\mathcal{V}(\mathbf{Z}_{12})$};
    \node[knode] (K2) at (2.75, 5.25) {$\mathcal{V}(\mathbf{Z}_2 \ltimes \mathbf{Z}_0)$};
    \node[knode] (K3) at (-0.5, 5.25) {$\mathcal{V}(\mathbf{Z}_3 \ltimes \mathbf{Z}_0)$};
    \node[knode] (K5) at (-3.2, 5.25) {$\mathcal{V}(\mathbf{Z}_5 \ltimes \mathbf{Z}_0)$};
    \node[knode] (K4) at (3.5, 6.5) {$\mathcal{V}(\mathbf{Z}_4 \ltimes \mathbf{Z}_0)$};
    \node[knode] (K6) at (0.75, 6.5) {$\mathcal{V}(\mathbf{Z}_6 \ltimes \mathbf{Z}_0)$};
    \node[knode] (K12) at (2.5, 7.5) {$\mathcal{V}(\mathbf{Z}_{12} \ltimes \mathbf{Z}_0)$};
    \node[knode] (K9) at (-1.5, 6.5) {$\mathcal{V}(\mathbf{Z}_9 \ltimes \mathbf{Z}_0)$};

    \node[znode] (Zdots_p_left) at (-4.25, 1.5) {$\cdots$};
    \node[knode] (Kdots_p_left) at (-4.5, 5.25) {$\cdots$};

    \draw (Z0) -- (Z2); 
    \draw (Z0) -- (Z3); 
    \draw (Z0) -- (Z5);
    \draw (Z2) -- (Z4); 
    \draw (Z2) to[bend left=10] (Z6); 
    \draw (Z3) to[bend right=10] (Z6);
    \draw (Z3) -- (Z9); 
    \draw (Z4) -- (Z12); 
    \draw (Z6) -- (Z12);
    \draw (K2) -- (K4); 
    \draw (K2) to[bend left=10] (K6); 
    \draw (K3) to[bend right=10] (K6);
    \draw (K3) -- (K9);
    \draw (K4) -- (K12); 
    \draw (K6) -- (K12);

    \draw[arrow] (Z2) to[bend right=5] (K2);
    \draw[arrow] (Z3) to[bend right=5] (K3);
    \draw[arrow] (Z4) to[bend right=10] (K4);
    \draw[arrow] (Z5) to[bend right=5] (K5);
    \draw[arrow] (Z6) to[bend right=5] (K6);
    \draw[arrow] (Z9) to[bend left=5] (K9);
    \draw[arrow] (Z12) to[bend left=10] (K12);
    \draw[dashed] (K5) to[bend right=5] (Top_p);
    \draw[dashed] (K9) to[bend left=5] (Top_p);
    \draw[dashed] (K12) to (Top_p);

    \draw[dashed] (Top_p) -- (GlobalTop);

    \draw[dashed] (Z0) -- (Zdots_p_left);
    \draw[arrow, dashed] (Zdots_p_left) to[bend right=5] (Kdots_p_left);
    \draw[dashed] (Kdots_p_left) to[bend left=5] (Top_p);

\end{scope}

\begin{scope}[xshift=4.5cm]
    \node (Top_n) at (2, 8) {$p\mathbb{AL}^-$};
    \node[znode] (Z-2) at (-2, 1.5) {$\mathcal{V}(\mathbf{Z}_{-2})$};
    \node[znode] (Z-3) at (0.5, 1.5) {$\mathcal{V}(\mathbf{Z}_{-3})$};
    \node[znode] (Z-5) at (3, 1.5) {$\mathcal{V}(\mathbf{Z}_{-5})$};
    \node[znode] (Z-4) at (-3.5, 2.75) {$\mathcal{V}(\mathbf{Z}_{-4})$};
    \node[znode] (Z-6) at (-0.75, 2.75) {$\mathcal{V}(\mathbf{Z}_{-6})$};
    \node[znode] (Z-9) at (1.5, 2.75) {$\mathcal{V}(\mathbf{Z}_{-9})$};
    \node[znode] (Z-12) at (-1.5, 4) {$\mathcal{V}(\mathbf{Z}_{-12})$};
    \node[knode] (K-2) at (-2.75, 5.25) {$\mathcal{V}(\mathbf{Z}_{-2} \ltimes \mathbf{Z}_0)$};
    \node[knode] (K-3) at (0.5, 5.25) {$\mathcal{V}(\mathbf{Z}_{-3} \ltimes \mathbf{Z}_0)$};
    \node[knode] (K-5) at (3.2, 5.25) {$\mathcal{V}(\mathbf{Z}_{-5} \ltimes \mathbf{Z}_0)$};
    \node[knode] (K-4) at (-3.5, 6.5) {$\mathcal{V}(\mathbf{Z}_{-4} \ltimes \mathbf{Z}_0)$};
    \node[knode] (K-6) at (-0.75, 6.5) {$\mathcal{V}(\mathbf{Z}_{-6} \ltimes \mathbf{Z}_0)$};
    \node[knode] (K-12) at (-2.5, 7.5) {$\mathcal{V}(\mathbf{Z}_{-12} \ltimes \mathbf{Z}_0)$};
    \node[knode] (K-9) at (1.5, 6.5) {$\mathcal{V}(\mathbf{Z}_{-9} \ltimes \mathbf{Z}_0)$};

    \node[znode] (Zdots_n_right) at (4.25, 1.5) {$\cdots$};
    \node[knode] (Kdots_n_right) at (4.5, 5.25) {$\cdots$};

    \draw (Z0) -- (Z-2); 
    \draw (Z0) -- (Z-3); 
    \draw (Z0) -- (Z-5);
    \draw (Z-2) -- (Z-4); 
    \draw (Z-2) to[bend right=10] (Z-6); 
    \draw (Z-3) to[bend left=10] (Z-6);
    \draw (Z-3) -- (Z-9); 
    \draw (Z-4) -- (Z-12); 
    \draw (Z-6) -- (Z-12);
    \draw (K-2) -- (K-4); 
    \draw (K-2) to[bend right=10] (K-6); 
    \draw (K-3) to[bend left=10] (K-6);
    \draw (K-3) -- (K-9);
    \draw (K-4) -- (K-12); 
    \draw (K-6) -- (K-12);

    \draw[arrow] (Z-2) to[bend left=5] (K-2);
    \draw[arrow] (Z-3) to[bend left=5] (K-3);
    \draw[arrow] (Z-4) to[bend left=10] (K-4);
    \draw[arrow] (Z-5) to[bend left=5] (K-5);
    \draw[arrow] (Z-6) to[bend left=5] (K-6);
    \draw[arrow] (Z-9) to[bend right=5] (K-9);
    \draw[arrow] (Z-12) to[bend right=10] (K-12);
        \draw[dashed] (K-9) to[bend right=5] (Top_n);
    \draw[dashed] (K-12) to (Top_n);
    \draw[dashed] (K-5) to[bend left=5] (Top_n);
    \draw[dashed] (Top_n) -- (GlobalTop);

    \draw[dashed] (Z0) -- (Zdots_n_right);
    \draw[arrow, dashed] (Zdots_n_right) to[bend left=5] (Kdots_n_right);
    \draw[dashed] (Kdots_n_right) to[bend left=5] (Top_n);
\end{scope}

\end{tikzpicture}
\caption{The lattice of join irreducible subvarieties of $p\mathbb{AL}$.}
\label{fig:combined-lattice-infinite-k9}
\end{figure}

\begin{proof}
    Let $\mathcal{V}$ be an arbitrary subvariety of $p\mathbb{AL}$. Using Theorem~\ref{t:Clifford} we can derive that finitely subdirectly irreducible pointed Abelian $\ell$-groups are totally ordered. Let us denote $\mathcal{V}_{\FSI}^+=\mathcal{V}_{\FSI} \cap p\mathbb{AL}^+$ and 
    $\mathcal{V}_{\FSI}^-=\mathcal{V}_{\FSI} \cap p\mathbb{AL}^-$.
    Thus we have $$\HSP(\mathcal{V}_{\FSI}^+ \cup \mathcal{V}_{\FSI}^-)=\HSP(\mathcal{V}_{\FSI})=\mathcal{V}.$$
    This shows that $\mathcal{V}$ is the join of $\HSP(\mathcal{V}_{\FSI}^+)$ and $\HSP(\mathcal{V}_{\FSI}^-)$. By Theorem~\ref{t:final positive} $\HSP(\mathcal{V}_{\FSI}^+)=\mathcal{V}_{I_1,J_1}$ for some sets $I_1,J_1$ and by Theorem~\ref{t:final negative} $\HSP(\mathcal{V}_{\FSI}^-)=\mathcal{V}^d_{I_2,J_2}$ for some sets $I_2,J_2$.
    Thus we have $\mathcal{V}=\mathcal{V}_{I_1,J_1} \lor \mathcal{V}^d_{I_2,J_2}$.

    We show that the equations from $S_{I_1,J_1} \cup S_{I_2,J_2}^d$ axiomatize the variety $\mathcal{V}$.
    First all the equations from $S_{I_1,J_1}$ are valid in all negatively pointed Abelian $\ell$-groups by Lemmas \ref{l:divisibity equation}, \ref{l:rank equation} and \ref{l:mix equation}. Similarly, all the equations from $S^d_{I_2,J_2}$ are valid in all positively pointed Abelian $\ell$-groups. Since all the equations from $S_{I_1,J_1}$ are valid in $\mathcal{V}_{I_1,J_1}$ and all the equations from $S^d_{I_2,J_2}$ are valid in $\mathcal{V}^d_{I_2,J_2}$, we conclude that all equations from $S_{I_1,J_1} \cup S_{I_2,J_2}^d$ are valid in $\mathcal{V}$.

    Now let us take an arbitrary $\A_a \in p\mathbb{AL}_{\FSI}$ satisfying all the equations from $S_{I_1,J_1} \cup S_{I_2,J_2}^d$. By Theorem~\ref{t:Clifford} $\A_a$ is totally ordered. We distinguish three cases:

    \begin{enumerate}
        \item If $a=0$ then $\A_a \in \mathcal{V}$ by Corollary~\ref{c:minimal variety}.
        \item If $a>0$ then $\A_a \in p\mathbb{AL}^+_{\FSI}$ and thus $\A_a \in \mathcal{V}_{I_1,J_1}$. Consequently $\A_a \in \mathcal{V}$.
        \item If $a<0$ then $\A_a \in p\mathbb{AL}^-_{\FSI}$ and thus $\A_a \in \mathcal{V}^d_{I_2,J_2}$. This implies $\A_a \in \mathcal{V}$.  
    \end{enumerate}
    
    This proves that the equations from $S_{I_1,J_1} \cup S_{I_2,J_2}^d$ axiomatize $\mathcal{V}$.
    \end{proof}


\section{Applying the Mundici functor} \label{s:Mundici}

In the previous sections, we developed a self-contained theory using the language of pointed Abelian $\ell$-groups.
In this section we will establish several connections between the theory of MV-algebras and the theory of pointed Abelian $\ell$-groups. We will show that the semantical description of subvarieties of $p\mathbb{AL}^+$ is analogous to the Komori classification of subvarieties of MV-algebras. We will conclude the section by deriving a complete classification of quasivarieties of $p\mathbb{AL}$ generated by totally ordered elements from the similar result about MV-algebras \cite{Gispert:UniversalClassesofMV-chains}.

Let us recall that MV-algebras can be understood as bounded intervals of Abelian $\ell$-groups. For us, an MV-algebra will be a structure of the following signature \tuple{\oplus,\otimes,\lor, \land,\neg,0,1}. For the basics of MV-algebras we refer the reader to any classical bibliography e.g.\ \cite{Cignoli-Ottaviano-Mundici:AlgebraicFoundations}.

We shall note that we are adding the symbols $\lor,\land$ and $\otimes$ into our language purely for our convenience. It is well-known that $\lor,\land$ and  $\otimes$ are term definable using the remaining operations.

The fundamental tool for working with MV-algebras is the Mundici functor, which is a categorical equivalence from the category of strongly positively pointed Abelian $\ell$-groups into the category of MV-algebras. As discussed in \cite{Young:Varieties_of_pointed_Abelian_l-groups}, one can generalize the Mundici functor to the whole category of positively pointed Abelian $\ell$-groups.
We will denote this extended functor as $\Gamma$ and define it as follows:

\begin{defn} \label{d:Gamma}
    Let $\A_u=\tuple{A,+,-,\lor, \land,0,u}$ be a positively pointed Abelian $\ell$-group and $u \geq 0$. Define the MV-algebra $\Gamma(\A_u)=\tuple{[0,u],\oplus,\otimes, \lor,\land,\neg,0,1}$, where the operations are defined as follows:
    \begin{enumerate}
        \item $a \oplus b:=(a+b) \land u$,
        \item $a \otimes b:=0 \lor (a+b-u)$,
        \item $\neg a:=u-a$,
        \item $1:=u.$
    \end{enumerate} 

For any positively pointed Abelian $\ell$-groups $\A_u$, $\B_v$ and a homomorphism  $f:\A_u \rightarrow \B_v$ we define $\Gamma(f):\Gamma(\A_u) \rightarrow \Gamma(\B_v)$ as $\Gamma(f):=f \restriction [0,u]$.
\end{defn}

It is well-known \cite{Cignoli-Ottaviano-Mundici:AlgebraicFoundations} that the Mundici functor ($\Gamma$ restricted to the subcategory of strongly positively pointed Abelian $\ell$-groups) is a categorical equivalence. 
By $\Gamma^{-1}$ we denote the corresponding inverse functor. In particular, we have $\Gamma( \Gamma^{-1}(\B)) \cong \B$ for each MV-algebra $\B$.
It is well known \cite{Young:Varieties_of_pointed_Abelian_l-groups} that any for positively pointed Abelian $\ell$-group $\A_a$ with $a \neq 0$ we have that $\Gamma^{-1}(\Gamma(\A_a))$ is equal to its strongly pointed convex $\ell$-subgroup. In particular, if $\A_a$ is itself strongly pointed, we have $\Gamma^{-1}(\Gamma(\A_a)) \cong \A_a$.

\begin{lemma} \label{l:preservation HSP}
    Let $\K \cup \{\A_u\}$ be a class of positively pointed Abelian $\ell$-groups. Then it holds
$\A_u \in \HSP(\K)$ implies $\Gamma(\A_u) \in \HSP(\Gamma[\K])$.
\end{lemma}

\begin{proof}
One has to show the following:
    \begin{enumerate}
    \item $\Gamma(\P(\K))=\P(\Gamma[\K])$ for any class $\K \subseteq p\mathbb{AL}^+$,
    \item $\A_u \in \S(\B_v)$ implies $\Gamma(\A_u) \in \S(\Gamma(\B_v))$ for any $\A_u,\B_v \in p\mathbb{AL}^+$,
    \item $\A_u \in \H(\B_v)$ implies $\Gamma(\A_u) \in \H(\Gamma(\B_v))$ for any $\A_u,\B_v \in p\mathbb{AL}^+$.
    \end{enumerate}
For detailed discussion, see \cite{Young:Varieties_of_pointed_Abelian_l-groups}.
\end{proof}
Using this observation we can easily provide the syntactical classification of all varieties of MV-algebras.

\begin{thm}[Komori] \label{t:Komori}
    Every non-trivial proper variety of MV-algebras is generated by $$\{\Gamma(\Z_i),\Gamma(\Z_j \ltimes \Z_0) \mid i \in I, j \in J\} \text{ for some finite sets } J \subseteq I.$$
    
\end{thm}

\begin{proof}
    Let $\mathcal{V}$ be a proper variety of MV-algebras. Since $\mathcal{V}$ is proper, there is $\A \notin \mathcal{V}$. Let us consider the variety of positively pointed Abelian $\ell$-groups $\K=\HSP(\Gamma^{-1}[\mathcal{V}])$. 
    By Lemma \ref{l:preservation HSP} we have $$\Gamma[\K]=\Gamma(\HSP(\Gamma^{-1}[\mathcal{V}])) \subseteq\HSP( \Gamma[\Gamma^{-1}[\mathcal{V}]])=\HSP(\mathcal{V})=\mathcal{V}.$$

    Since $\A \notin \mathcal{V}$, also $\Gamma (\Gamma^{-1}(\A)) \notin \mathcal{V}$ and by Lemma \ref{l:preservation HSP} we obtain $\Gamma^{-1}(\A) \notin \K$. Therefore, $\K$ is a proper subvariety of $p\mathbb{AL}^+$ and thus by Theorem \ref{t:final positive} we have $\HSP(\Gamma^{-1}[\mathcal{V}])=\mathcal{V}_{I,J}$ for some finite sets $J\subseteq I$.
    Since $\Z_i,\Z_j \ltimes \Z_0 \in \mathcal{V}_{I,J}$ for $i \in I, j \in J$, by Lemma \ref{l:preservation HSP} we obtain $\Gamma(\Z_i), \Gamma(\Z_j \ltimes \Z_0) \in \Gamma[\mathcal{V}_{I,J}] \subseteq \mathcal{V}$ for $i \in I$ and $j \in J$.
   
    Now it remains to check $\mathcal{V}$ is generated by $\Gamma(\Z_i), \Gamma(\Z_j \ltimes \Z_0) $ for every $i \in I, j \in J$. Let $\B \in \mathcal{V}$. We obtain $\Gamma^{-1}(\B) \in \Gamma^{-1}[\mathcal{V}] \subseteq \HSP(\Gamma^{-1}[\mathcal{V}])=\mathcal{V}_{I,J}=\HSP(\Z_i,\Z_j \ltimes \Z_0 \mid i \in I, j\in J).$ By applying $\Gamma$ it follows $\B \in \HSP(\Gamma(\Z_i),\Gamma(\Z_j \ltimes \Z_0) \mid i \in I, j\in J).$ This completes the proof.
\end{proof}

Here, we used the classification of subvarieties of $p\mathbb{AL}^+$ to obtain the classification of all subvarieties of MV-algebras. We should note that we could proceed in the opposite way: using Komori's classification, Corollary~\ref{c:minimal variety}, and Theorem~\ref{t:strong_unit}, we can derive the semantic part of Theorem~\ref{t:final final}. For further details, see the discussion in \cite{Young:Varieties_of_pointed_Abelian_l-groups}.

\subsection{Gispert's classification of universal classes of totally ordered MV-algebras}
In the rest of this section we focus on quasivarieties and universal classes.
Let $\mathbf S$ denote any finitely generated dense $\ell$-subgroup of $\mathbf R$ such that $\mathbf S \cap \mathbf Q=\mathbf Z$.
Recall the following result:

\begin{thm}[Gispert {\cite[Theorem 3.1]{Gispert:UniversalClassesofMV-chains}}] \label{t:Gispert}
    Let $\mathbb M$ be a class of totally ordered MV-algebras. Then the following properties are equivalent:

    \begin{enumerate}
        \item $\mathbb M$ is universal.
         \item There exist $I,J,K \subseteq \N \setminus \{0\}$ and for every $j \in J$, a nonempty subset $D_j \subseteq \div(j)$ such that $\mathbb M$ is equal to
        $$\ISPU(\{\Gamma(\Z_i) \mid i \in I\} \cup \{\Gamma(\Z_j \ltimes \Z_{d_j}) \mid j \in J, d_j \in D_j\} \cup \{\Gamma(\mathbf S_k) \mid k \in K\}).$$
    \end{enumerate}
    
\end{thm}

We need to show how $\Gamma$ behaves with respect to ultraproducts. 

\begin{lemma} \label{l:Mundici_embedding}
Let $\K \cup \{\A_a\}$ be a class of strongly positively pointed Abelian $\ell$-groups.
 Then we have $\A_a \in \ISPU(\K)$ iff $\Gamma(\A_a) \in \ISPU(\Gamma[\K])$.
\end{lemma}

\begin{proof}
    Throughout this proof, let $I$ be an arbitrary fixed set, $\U$ be an ultrafilter over $I$ and $\{\A^i_{a_i} \mid i \in I\} \subseteq \K$.
    First, let us note that $\Gamma(\prod_{i \in I}\A^i_{a_i}/\U)=\Gamma(\prod_{i \in I}\A^i_{a_i})/\U=(\prod_{i \in I}\Gamma(\A^i_{a_i})/\U)$. Consequently, $\A_a \in \ISPU(\K)$ implies $\Gamma(\A_a) \in \ISPU(\Gamma[\K])$.

    For the other inclusion first assume $\Gamma(\A_a) \in \IS(\prod_{i \in I} \Gamma(\A^i_{a_i})/\U)$. Let $\B_b$ denote the strongly pointed convex $\ell$-subgroup of $\prod_{i \in I} \A^i_{a_i}/\U$. We know that we have $\Gamma^{-1}( \Gamma(\prod_{i \in I}\A^i_{a_i}/\U))=\B_b$. Since $\Gamma^{-1}$ preserves embeddings, we have $\Gamma^{-1} (\Gamma(\A_a)) \in \IS (\Gamma^{-1} (\prod_{i \in I} \Gamma(\A^i_{a_i})/\U))$. Since $\A_a \cong \Gamma^{-1} (\Gamma(\A_a))$ and $$\Gamma^{-1} (\prod_{i \in I} \Gamma(\A^i_{a_i})/\U)=\Gamma^{-1} (\Gamma(\prod_{i \in I} \A^i_{a_i}/\U))=\B_b,$$
    we obtain $\A_a \in \IS(\B_b)$.
    Since $\B_b \in \IS(\prod_{i \in I} \A^i_{a_i}/\U)$, it follows that $\A_a \in \IS(\prod_{i \in I} \A^i_{a_i}/\U)$ and consequently we have $\A_a \in \ISPU(\K)$. This completes the proof.
\end{proof}

\begin{thm} \label{t:Gispert-groups}
    Let $\mathbb M$ be a class of totally ordered positively pointed Abelian $\ell$-groups. Then the following properties are equivalent:

    \begin{enumerate}
        \item $\mathbb M$ is universal.
        \item There exist $I \subseteq \N$,  $J,K \subseteq \N \setminus \{0\}$ and for every $j \in J$, a nonempty subset $D_j \subseteq \div(j)$ such that
        $$\mathbb M=\ISPU(\{\Z_i \mid i \in I\} \cup \{\Z_j \ltimes \Z_{d_j} \mid j \in J, d_j \in D_j\} \cup \{\mathbf S_k \mid k \in K\}).$$
        
    \end{enumerate}
\end{thm}
\begin{proof}
    The implication $2$ to $1$ is trivial. To prove the other implication assume $\mathbb M$ is universal. Therefore, by Theorem~\ref{t:partembISPU}  we have $\mathbb M=\ISPU(\mathbb M)$. Let us denote by $\K$ the subclass of $\mathbb M$ consisting of all strongly pointed Abelian $\ell$-groups.
    
    Using Lemma~\ref{l:Mundici_embedding} it follows $\Gamma[\K]=\ISPU(\Gamma[\K])$.
    Thus $\Gamma[\K]$ is universal and hence by Theorem~\ref{t:Gispert} we have that there exist $I,J,K \subseteq \N \setminus \{0\} $ and for every $j \in J$, a nonempty subset $D_j \subseteq \div(j)$ such that
        $$\Gamma[\K] \hspace{-1mm}= \hspace{-0.9mm}\ISPU(\{\Gamma(\Z_i) \mid i \in I\} \cup \{\Gamma(\Z_j \ltimes \Z_{d_j}) \mid j \in J, d_j \in D_j\} \cup \{\Gamma(\mathbf S_k) \mid k \in K\}).$$
    Now, again using Lemma~\ref{l:Mundici_embedding} we obtain 
    $$\K = \ISPU(\{\Z_i \mid i \in I\} \cup \{\Z_j \ltimes \Z_{d_j} \mid j \in J, d_j \in D_j\} \cup \{\mathbf S_k \mid k \in K\}) \cap sp\mathbb{AL}^+.$$
    We discuss two options:

    \begin{enumerate}[leftmargin=*]
        \item If $\Z_0 \in \mathbb M$ then $p\mathbb{AL}_{\FSI}^0 \subseteq \mathbb M$ by Lemma~\ref{l:about R1} and thus by Theorem~\ref{t:strong_unit} we have 
        \begin{align*} \hspace{-1mm} 
            \mathbb M&=\ISPU(\K) \cup p\mathbb{AL}_{\FSI}^0=\ISPU(\K) \cup \ISPU(\Z_0)\\
        &=\ISPU(\{\Z_i \mid i \in I \cup \{0\}\} \cup \{\Z_j \ltimes \Z_{d_j} \mid j \in J, d_j \in D_j\} \cup \{\mathbf S_k \mid k \in K\}).
        \end{align*}
        \item If $\Z_0 \notin \mathbb M$, we obtain by Lemma~\ref{l:about R1} that $p\mathbb{AL}_{\FSI}^0 \cap \mathbb M=\emptyset$ and thus 
            \begin{align*} \hspace{-1mm}
            \mathbb M&=\ISPU(\K)\\&=
            \ISPU(\{\Z_i \mid i \in I\} \cup \{\Z_j \ltimes \Z_{d_j} \mid j \in J, d_j \in D_j\} \cup \{\mathbf S_k \mid k \in K\}).
        \end{align*}
           \end{enumerate}
    This completes the proof.
\end{proof}

\begin{cor} \label{c:Gispert-groups}
    Let $\K$ be the quasivariety generated by a class of totally ordered positively pointed Abelian $\ell$-groups. Then there exist $I \subseteq \N$, $J,K \subseteq \N \setminus \{0\}$ and for every $j \in J$, a nonempty subset $D_j \subseteq \div(j)$ such that
        $$\K=\ISPPU(\{\Z_i \mid i \in I\} \cup \{\Z_j \ltimes \Z_{d_j} \mid j \in J, d_j \in D_j\} \cup \{\mathbf S_k \mid k \in K\}).$$
   
\end{cor}
Let us note that statements analogous to Theorem~\ref{t:Gispert-groups} and Corollary~\ref{c:Gispert-groups} can be proved also for negatively pointed Abelian $\ell$-groups. Therefore, we can generalize Theorem \ref{t:Gispert-groups} and Corollary \ref{c:Gispert-groups} as follows.

\begin{cor} \label{c:universal classes}
    Let $\mathbb M$ be a class of totally ordered pointed Abelian $\ell$-groups. Then the following properties are equivalent:

    \begin{enumerate}
        \item $\mathbb M$ is universal.
        \item There exist $I \subseteq \mathbb Z$, $J,K \subseteq \mathbb Z \setminus \{0\}$ and for every $j \in J$, a nonempty subset $D_j \subseteq \div(j)$ such that
        $$\mathbb M=\ISPU(\{\Z_i \mid i \in I\} \cup \{\Z_j \ltimes \Z_{d_j} \mid j \in J, d_j \in D_j\} \cup \{\mathbf S_k \mid k \in K\}).$$
    \end{enumerate}
\end{cor}

\begin{cor} \label{c:Gispert-groups ALL}
    Let $\K$ be the quasivariety generated by a class of totally ordered pointed Abelian $\ell$-groups. Then there exist $I \subseteq \mathbb Z$, $J,K \subseteq \mathbb Z \setminus \{0\}$ and for every $j \in J$, a nonempty subset $D_j \subseteq \div(j)$ such that
       
        $$\K=\ISPPU(\{\Z_i \mid i \in I\} \cup \{\Z_j \ltimes \Z_{d_j} \mid j \in J, d_j \in D_j\} \cup \{\mathbf S_k \mid k \in K\}).$$
\end{cor}





\section{Future Work}\label{s:Conclusions}

In this section, we outline possible directions for future research on the topics covered in this paper.

\begin{itemize}
    \item In this paper, we discuss the characterization of all subquasivarieties generated by totally ordered $\ell$-groups. However, we provide only the structural characterization of these subquasivarieties without any quasiequational basis. The obvious future step is to provide these axiomatizations, which we can obtain by analyzing the algebras $\Z_i$, $\Z_j \ltimes \Z_{d_j}$ and $\mathbf S_k$ from Corollary \ref{c:Gispert-groups ALL}.

    \item One of the main motivations for this paper was to show that when we discuss the subvarieties of $p\mathbb{AL}$, we can do so purely in the theory of pointed Abelian $\ell$-groups in a self-contained way without any circular reasoning. However, the same still needs to be done regarding the characterization of subquasivarieties of $p\mathbb{AL}$ that are generated by their totally ordered members.

    \item One of the most significant drawbacks of Corollary \ref{c:strong unit quasi} is that we require the strongly pointed $\ell$-group to be \emph{totally ordered}. It is known that there are subquasivarieties of $p\mathbb{AL}$ that are not generated by their totally ordered members and they have no immediate counterpart in MV-algebras. For example, consider the quasivariety given by the quasiequation $\fa \geq 0 \implies x \geq 0$ (see \cite{Cintula-Jankovec-Noguera:SuperabelianLogics}), which is known to not be generated by its totally ordered members.
    We do not aim to provide an explicit description of all subquasivarieties of $p\mathbb{AL}$, as the lattice of these subquasivarieties is likely Q-universal, similarly to the class of all subquasivarieties of MV-algebras \cite{Adams-Dziobiak:Q-universal}. Instead, we aim to explicitly describe the additional subquasivarieties that do not correspond to any subquasivariety of MV-algebras.

    \item An additional question is connected to expanding the language of pointed Abelian $\ell$-groups with a modal operator $\tuple{\A_a,\Box}$ \cite{Jankovec-Poiger:ModalAbelianLogic}. By doing so, we obtain a variety of modal pointed Abelian $\ell$-groups that is not generated by totally ordered algebras. As such, we cannot utilize Theorem \ref{t:strong_unit} here, meaning we need to reformulate more general criteria for when a pointed modal Abelian $\ell$-group is in a variety generated by its strongly pointed modal convex $\ell$-subgroup.
\end{itemize}



\section*{Declarations}

\begin{description}
    \item[Ethical Approval] Not applicable.
    \item[Competing interests] The author declares no competing interests.
    \item[Authors' contributions] Not applicable.
    \item[Funding] This work was supported by the Czech Science Foundation grant {25-17958J}. This work was also financially supported by the Charles University Grant Agency, project no. 101724, entitled "Exploring the foundations of reasoning for rational interaction under conditions of uncertainty", implemented at the Faculty of Arts of Charles University.
    \item[Availability of data and materials] Not applicable.
\end{description}


\bibliographystyle{plainurl}
\bibliography{mfl}

\end{document}